\documentclass[11pt]{article}
\usepackage{amsfonts}
\usepackage{graphicx}
\usepackage{amsmath}
\usepackage{amsthm}
\usepackage{amssymb}
\usepackage{amscd}
\usepackage{graphicx}
\usepackage{microtype}
\usepackage{enumerate}
\usepackage{fullpage}
\usepackage{pinlabel}
\usepackage[top=1.1 in,bottom=1.4in,left=1in, right=1in]{geometry}
\usepackage{extpfeil}

\theoremstyle{definition}
\newtheorem{theorem}{Theorem}[section]

\newtheorem{claim}[theorem]{Claim}

\newtheorem{conjecture}[theorem]{Conjecture}
\newtheorem{corollary}[theorem]{Corollary}

\newtheorem{example}[theorem]{Example}

\newtheorem{lemma}[theorem]{Lemma}
\newtheorem{notation}[theorem]{Notation}

\newtheorem{proposition}[theorem]{Proposition}
\newtheorem{remark}[theorem]{Remark}

\newtheorem{definition}[theorem]{Definition}

\newcommand{\R}{\mathbb{R}}

\newcommand{\Z}{\mathbb{Z}}
\newcommand{\Q}{\mathbb{Q}}

\newcommand{\boldhead}[1]{%
 {\bigskip \noindent \bfseries #1 \\ }}

\renewcommand{\ldots}{...}

\DeclareMathOperator{\Hom}{Hom}
\DeclareMathOperator{\rot}{rot}
\DeclareMathOperator{\PSL}{PSL}
\DeclareMathOperator{\PSLk}{PSL^{(k)}}

\DeclareMathOperator{\SL}{SL}
\DeclareMathOperator{\rott}{\tilde{r}ot}
\DeclareMathOperator{\Diff}{Diff}
\DeclareMathOperator{\Homeo}{Homeo}

\DeclareMathOperator{\id}{id}
\DeclareMathOperator{\euler}{e}
\DeclareMathOperator{\proj}{proj}
\DeclareMathOperator{\fix}{fix}
\DeclareMathOperator{\Aut}{Aut}

\title{Components of spaces of surface group representations }
\author{Kathryn Mann}
\date{}

\begin{document}

\maketitle

\abstract{We give a new lower bound on the number of topological components of the space of representations of a surface group into the group of orientation preserving homeomorphisms of the circle.   Precisely, for the fundamental group of a genus $g$ surface, we show there are at least $k^{2g} + 1$ components containing representations with Euler number $\frac{2g-2}{k}$, for each nontrivial divisor $k$ of $2g-2$.  
We also show that certain representations are \emph{rigid}, meaning that all deformations lie in the same semi-conjugacy class.   Our methods apply to representations of surface groups into finite covers of $\PSL(2,\R)$ and into $\Diff_+(S^1)$ as well, in which case we recover theorems of W. Goldman and J. Bowden. 

The key technique is an investigation of \emph{local maximality} phenomena for rotation numbers of products of circle homeomorphisms using techniques of Calegari--Walker.  This is a new approach to studying deformation classes of group actions on the circle, and may be of independent interest.}


\section{Introduction}

Let  $\Gamma_g$ denote the fundamental group of the closed, genus $g$ surface $\Sigma_g$, and let $G$ be a group of orientation preserving homeomorphisms of the circle.  We study the space $\Hom(\Gamma_g, G)$ of representations from $\Gamma_g$ to $G$.  
This space has two natural and important interpretations; first as the space of flat circle bundles over $\Sigma_g$ with structure group $G$, and secondly as the space of actions of $\Gamma_g$ on $S^1$ with degree of regularity specified by $G$.  (For example, taking $G = S^1$ corresponds to actions by isometries, $G= \PSL(2,\R)$ to actions by M\"obius transformations, and $G = \Diff^r(S^1)$ to actions by $C^r$ diffeomorphisms.)
A fundamental question is to identify and characterize the \emph{deformation classes} of $\Gamma_g$ actions; equivalently, the deformation classes of flat circle bundles, or the connected components of $\Hom(\Gamma_g, G)$.

When $G \subset \Homeo_+(S^1)$ is a transitive Lie group, $\Hom(\Gamma_g, G)$ is a classical object of study called the \emph{representation variety}.  In this case, the Milnor-Wood inequality together with work of W. Goldman gives a complete classification of the connected components of $\Hom(\Gamma_g, G)$ using the Euler number.
Far less is known when $G$ is not a Lie group, and the problem of classifying components of $\Hom(\Gamma, G)$ for $G = \Homeo_+(S^1)$ is  essentially completely open.   The purpose of this paper is to develop a new approach to distinguish components of $\Hom(\Gamma_g, G)$ applicable to the $G = \Homeo_+(S^1)$ case.   

Our approach is based on recent work of Calegari--Walker, which gives tools to analyze the rotation numbers of products of circle homeomorphisms. With this approach, we are able to give a new lower bound on the number of connected components of $\Hom(\Gamma_g, \Homeo_+(S^1))$, show that the Euler number does \emph{not} distinguish components (see Theorem \ref{main thm}), and prove strong rigidity results (Theorem \ref{semiconj thm}) for certain representations.

\boldhead{Lie groups: known results}
We start with a brief summary of known results on $\Hom(\Gamma_g, G)$ when $G \subset \Homeo(S^1)$ is a Lie group.  In this case, $\Hom(\Gamma_g, G)$ has the structure of an affine variety and hence can have only finitely many components.  The interesting case is when $G$ acts transitively on $S^1$, in which case $G$ must either be $S^1$ or the $k$-fold cyclic cover $\PSLk$ of $\PSL(2,\R)$ for some $k \geq 1$.  
It is easy to see that $\Hom(\Gamma_g, S^1)$ is connected, and the components of $\Hom(\Gamma_g, \PSLk)$
are completely classified by the following theorem.  

\begin{theorem}[Goldman \cite{Goldman}, using also Milnor \cite{Milnor}] \label{goldman thm}
Let $G$ be the $k$-fold cyclic cover $\PSLk$ of $\PSL(2,\R)$, $k \geq 1$. The connected components of $\Hom(\Gamma_g, G)$ are completely classified by 
\begin{enumerate}[a)]
\item The Euler number of the representation, which assumes all values between $-\lfloor \frac{2g-2}{k} \rfloor$ and $\lfloor \frac{2g-2}{k} \rfloor$ and is constant on connected components.  Each value of the Euler number is assumed on a single connected component, with the sole exception of b) below.
\item  In the case where $2g-2 = nk$ for some integer $n$, there are $k^{2g}$ connected components of representations with Euler number $n$ and  $k^{2g}$ with Euler number $-n$.  These components are distinguished by the rotation numbers of a standard set of generators for $\Gamma_g$.  
\end{enumerate}
\end{theorem}

\noindent The \emph{Euler number} of a representation $\rho: \Gamma_g \to G \subset \Homeo_+(S^1)$ and the \emph{rotation number} of a homeomorphism of $S^1$ are classical invariants; we recall their definitions in Section \ref{background sec}.

\boldhead{Non-Lie groups}
When $G$ is not a Lie group, describing the space $\Hom(\Gamma_g, G)$ is more challenging.  We are particularly interested in the case $G = \Homeo_+(S^1)$.  
Though there are still only finitely many possible values for the Euler number (this is the \emph{Milnor-Wood inequality}, see Section \ref{Euler subsec}), it is not even known if $\Hom(\Gamma_g, \Homeo_+(S^1))$ has finitely many or infinitely many components.    
On the other hand, it is possible \emph{a priori} that $\Hom(\Gamma_g, \Homeo_+(S^1))$ could be in a sense ``more connected" than $\Hom(\Gamma_g, \PSLk)$ -- for instance two representations into $\PSLk$, both with Euler number $\frac{2g-2}{k}$, but lying in different components of $\Hom(\Gamma_g, \PSLk)$, could potentially be connected by a path of representations in $\Hom(\Gamma_g, \Homeo_+(S^1))$.  

J. Bowden recently showed that this kind of additional connectedness does not hold for $\Hom(\Gamma_g, \Diff_+(S^1))$.  He proves the following.  

\begin{theorem}[Bowden, see Theorem 9.5 in \cite{Bowden}] \label{bowden thm}
Let $\rho_1$ and $\rho_2: \Gamma_g \to \PSLk$ be representations that lie in different connected components of $\Hom(\Gamma_g, \PSLk)$.  Then they also lie in different connected components of $\Hom(\Gamma_g, \Diff_+(S^1))$.  
\end{theorem}

Bowden gives two proofs, one using invariants of contact structures associated to the transverse foliation on a flat circle bundle, and the other using structural stability of Anosov flows.  Both proofs assume $C^\infty$ regularity of diffeomorphisms, although a similar strategy might work assuming only $C^2$.  However, the question for representations into $\Homeo_+(S^1)$ is essentially different, and Bowden asks if his results hold in this case. 
Our main theorem gives an affirmative answer.

\begin{theorem}[\textbf{Lower bound}] \label{main thm}
Let $\Gamma_g$ be the fundamental group of a genus $g$ surface.  For each nontrivial divisor $k$ of $2g-2$, there are at least $k^{2g} +1$ components of $\Hom(\Gamma_g, \Homeo_+(S^1))$ consisting of representations with Euler number $\frac{2g-2}{k}$.   

\noindent In particular, two representations into $\PSLk$ that lie in different components of $\Hom(\Gamma_g, \PSLk)$ necessarily lie in different components of $\Hom(\Gamma_g, \Homeo_+(S^1))$.  
\end{theorem} 

The primary tool in our proof, suggested to the author by D. Calegari, is the study of \emph{rotation numbers} of elements in the image of a representation $\rho$.  For each $\gamma \in \Gamma_g$, let $\rot_\gamma: \Hom(\Gamma_g, \Homeo_+(S^1)) \to \R/\Z$ be defined by  $\rot_\gamma(\rho) = \rot(\rho(\gamma))$, where $\rot(\cdot)$ denotes rotation number.  We prove a strong form of rigidity.  

\begin{theorem}[\textbf{Rotation number rigidity}]  \label{rot rig thm}
Let $X \subset \Hom(\Gamma_g, \Homeo_+(S^1))$ be the connected component of a representation $\rho$ with image in $\PSLk$ and Euler number $\euler(\rho) = \pm( \frac{2g-2}{k})$.  Then $\rot_\gamma$ is constant on $X$. 
\end{theorem} 

Bowden calls a representation $\Gamma_g \to \PSLk$ with Euler number $\frac{2g-2}{k}$ an \emph{Anosov representation}, as the transverse foliation on the associated flat circle bundle is the weak-stable foliation of an Anosov flow.  A key step in Bowden's second proof of Theorem \ref{bowden thm} is to show that any representation $\Gamma_g \to \Diff_+(S^1)$ in the connected component of an Anosov representation is also Anosov, in that it preserves the weak-stable foliation of some Anosov flow.  From this, he is able to show that each Anosov component of $\Hom(\Gamma_g, \Diff_+(S^1))$ consists of a single conjugacy class of representations, using results of Matsumoto and Ghys.  

We reach a similar conclusion, but the appropriate notion for homeomorphisms is \emph{semi-conjugacy}. 

\begin{theorem}[\textbf{Representation rigidity}] \label{semiconj thm}
Let $\rho: \Gamma_g \to \PSLk$ have Euler number $\pm(\frac{2g-2}{k})$.  Then the connected component containing $\rho$ in $\Hom(\Gamma_g, \Homeo_+(S^1))$ consists of the semi-conjugacy class of $\rho$. 
\end{theorem}
\noindent The definition of semi-conjugacy class is given Section \ref{semiconj sec}.

Translated into the language of foliations, Theorem \ref{semiconj thm} says in particular that the condition of the transverse foliation on the flat circle bundle associated to $\rho$ being semi-conjugate to an Anosov foliation is an open condition -- a foliation $C^0$ close to such will still be semi-conjugate to an Anosov foliation.  Compare Theorem 9.5 in \cite{Bowden}. 

\boldhead{Outline}
We begin with some background on rotation numbers and the work of Calegari and Walker in \cite{CW}.  This leads us to a definition of the Euler number in the language of rotation numbers, the Milnor-Wood inequality, and a description of the dynamics of representations into $\PSL(2,\R)$ and $\PSLk$ with maximal Euler number.  
\smallskip 

In Section \ref{proof strategy sec} we outline our proof strategy for Theorem \ref{rot rig thm}, motivating it by giving the proof of a toy case.   Section \ref{CW sec} consists of a detailed study of rotation numbers of products of homeomorphisms, with the algorithm of Calegari--Walker as our main tool.  We focus on examples that will later play a role in the proof of Theorem \ref{rot rig thm}. \smallskip

The proof of Theorem \ref{rot rig thm} is carried out in Section \ref{rot pf sec}.  
Section \ref{technical subsec} is somewhat technical, and a reader interested in only the broad flavor of the proof of our main theorems may wish to skip the proofs here on a first reading.  However, the reader with an interest in rotation numbers as a tool for parametrizing or studying representation spaces should find that Section \ref{technical subsec} contains some interesting techniques.   Section \ref{comm product subsec} also uses the techniques developed in Section \ref{CW sec} to study rotation numbers of products of homeomorphisms, but the proofs here are much quicker.  

In Section \ref{rot rig subsec} we prove a narrower form of rotation rigidity.  Extending this to the general result requires an \emph{Euclidean algorithm for rotation numbers of commutators}, which we develop in Section \ref{Euc alg subsec}.  The reader may (again) either find this \emph{Euclidean algorithm} technique to be of independent interest, or may choose to skip it on a first reading.\smallskip

Finally, in Section \ref{main pf sec} we deduce Theorem  \ref{main thm}  from Theorem \ref{rot rig thm} using a trick of Matsumoto, and in Section \ref{semiconj sec} we discuss semi-conjugacy and derive Theorem \ref{semiconj thm} from our earlier work using results of Matsumoto and Ghys.    Section \ref{sharpness sec} gives evidence for (as well as a conjecture on) the sharpness of our main theorems.

\boldhead{Acknowledgements}
The author thanks Danny Calegari for suggesting that rotation numbers might distinguish components of $\Hom(\Gamma_g, \Homeo_+(S^1))$, and explaining the idea of rotation numbers as ``coordinates" on representation spaces.  We are also grateful to Jonathan Bowden, Benson Farb, Shigenori Matsumoto, and Alden Walker for many helpful conversations and suggestions regarding this work.

\section{Background} \label{background sec}

\subsection{Rotation numbers} \label{rot number subsec}
Let $\Homeo_\Z(\R)$ denote the group of orientation-preserving homeomorphisms of $\R$ that commute with integral translations.  

\begin{definition}
Consider $S^1$ as $\R/\Z$, and let $x \in S^1$.  The  $\R/\Z$-valued \emph{rotation number} of an element $g \in \Homeo_+(S^1)$ is given by 
$$\rot(g) := \lim \limits_{n \to \infty} \frac{\tilde{g}^n(x)}{n} \, \mod \Z$$  
where $\tilde{g}$ is any lift of $g$ to $\Homeo_\Z(\R)$.
\end{definition} 
This limit always exists, and is independent of the choice of lift $\tilde{g}$ and choice of point $x \in S^1$.  A good exposition of the basic theory can be found in \cite{Navas} or \cite{Ghys groups acting}.   One easy fact that we will make use of is that a homeomorphism has a periodic point of period $k$ if and only if it has a rotation number of the form $m/k$ for some $m \in \Z$.  

We can define a $\Z$-valued \emph{lifted rotation number} (often also called ``translation number") for elements of $\Homeo_\Z(\R)$ in the same way.
\begin{definition} 
Let $\tilde{g} \in \Homeo_\Z(\R)$ and $x \in \R$.  The \emph{lifted rotation number} $\rott(\tilde{g})$ is given by
$$\rott(\tilde{g}) := \lim \limits_{n \to \infty} \frac{\tilde{g}^n(x)}{n} $$
\end{definition}

Again, the limit always exists, is finite, and is independent of the choice of point $x$.   
  
\boldhead{Lifted commutators}
The rotation number and lifted rotation number are closely related.  If $g \in \Homeo_+(S^1)$, and $\tilde{g}$ is a lift of $g$ to $\Homeo_\Z(\R)$, then $\rot(g) \equiv \rott(\tilde{g})  \mod \Z$, and different choices of lifts of $g$ change the value of $\rott(\tilde{g})$ by an integer.
However, for a \emph{commutator} $g = [a,b] : = aba^{-1}b^{-1}$, the homeomorphism $[\tilde{a}, \tilde{b}] \in \Homeo_\Z(\R)$ is independent of choice of lifts $\tilde{a}$ and $\tilde{b}$ of $a$ and $b$.  We will henceforth use the notation $\rott[a,b]: = \rott([\tilde{a}, \tilde{b}])$ to denote its lifted rotation number, and refer to $[\tilde{a}, \tilde{b}]$ as a \emph{lifted commutator}.  Similarly, for a product of commutators we define $\rott \left( \prod[a_i, b_i] \right) := \rott(\prod [\tilde{a_i}, \tilde{b_i}])$; this also independent of the choice of lifts. 
\bigskip

Both $\rot$ and $\rott$ are continuous with respect to the uniform norms on $\Homeo_+(S^1)$ and $\Homeo_\Z(\R)$, are invariant under (semi-)conjugacy, and are homomorphisms when restricted to cyclic subgroups (i.e. $\rott(g^n) = n\rott(g)$).  However, $\rot$ and $\rott$ are \emph{not} homomorphisms in general -- in fact it is easy to produce examples of elements of $\Homeo_\Z(\R)$ with lifted rotation number zero whose product has lifted rotation number 1. 

In \cite{CW}, Calegari and Walker develop an approach to studying the possible values of $\rott(w)$ for a word $w$ in the free semigroup generated by $\tilde{a}$ and $\tilde{b}$.  In particular, they give an algorithm to determine the maximum value of $\rott(w)$ given the lifted rotation numbers of $\tilde{a}$ and $\tilde{b}$.   We will use an extension of this algorithm later in Sections \ref{CW sec} and  \ref{rot pf sec}.  For now, the main result that we need is the following.  

\begin{theorem}[Theorem 3.9 in \cite{CW}] \label{CW thm}
Suppose that $\rott(\tilde{a}) = r$ and $\rott(\tilde{b}) = s$ for some $r, s \geq 0$.  Then the maximum possible value of $\rott(\tilde{a}\tilde{b})$ is given by 
$$\sup \limits_{p_1/q\leq r, \,\, p_2/q \leq s} \frac{p_1 + p_2 + 1}{q}$$
where $p_i \geq 0$ and $q > 0$ are integers.  
\end{theorem}

Calegari and Walker also show that there are strong restrictions on the possible values of the lifted rotation number of a commutator.  
The following result will play a crucial role in our proofs. 
\begin{lemma}[Example 4.9 in \cite{CW}] \label{comm lemma}
Let $a, b \in \Homeo_+(S^1)$.  The following hold. 
\begin{enumerate}[i)]
\item If $\rot(a) \notin \Q$ or $\rot(b) \notin \Q$, then $\rott [a,b]= 0$.
\item If $\rot(a) = p/q$ or $\rot(b) = p/q$, where $p/q \in \Q$ is in lowest terms, then $|\rott [a,b] | \leq 1/q$.
\end{enumerate}
\end{lemma}

We conclude this section with an elementary calculation that we will also use later.    
\begin{lemma} \label{additivity lemma}
Let $\tilde{f}$ and $\tilde{g} \in \Homeo_\Z(\R)$ satisfy $\tilde{f}\tilde{g} = T^n$ where $T^n$ denotes the translation $T^n(x) := x+n$, for $n \in \Z$.   Then $\rott(\tilde{f}) + \rott(\tilde{g}) = n$.  
\end{lemma}  

\begin{proof} 
That $\tilde{f}\tilde{g} = T^n$ implies that $\tilde{f} = T^n \circ \tilde{g}^{-1}$, hence 
$$\rott(\tilde{f}) = \rott(T^n \circ \tilde{g}^{-1})$$
Since $T^n$ is a translation commuting with $\tilde{g}^{-1}$, it follows from the definition of rotation number that $\rott(T^n \circ \tilde{g}^{-1}) = \rott(T^n) + \rott(\tilde{g}^{-1}) = n - \rott(\tilde{g})$,
as claimed.  

\end{proof} 

\subsection{The Euler number} \label{Euler subsec}

Classically, the \emph{Euler number} of a representation $\rho: \Gamma_g \to \Homeo_+(S^1)$ is defined in terms of characteristic classes --  it is the result of evaluating the pullback $\rho^*(e_\Z) \in H^2(\Gamma_g; \Z)$ of the canonical Euler class $e_\Z \in H^2(\Homeo_+(S^1); \Z)$ on the fundamental class $[\Gamma_g] \in H_2(\Gamma_g; \Z)$.   However, we will use the following alternative definition which emphasizes the relationship of the Euler number to the lifted rotation number.  This idea is originally due to Milnor in \cite{Milnor}, and is made explicit in \cite{Wood}.  

\begin{definition}[Euler number] \label{euler def}
Let $\{a_1, b_1, ... a_g, b_g\}$ be a \emph{standard set of generators} for $\Gamma_g$, meaning that $\Gamma_g$ has the presentation 
$$\Gamma_g  = \langle a_1, b_1, ... a_g, b_g | \,\,[a_1, b_1][a_2, b_2]...[a_g, b_g]\rangle.$$
Let $\rho: \Gamma_g \to \Homeo_+(S^1)$ be a representation.  We define the \emph{Euler number} $\euler(\rho)$ by 
$$\euler(\rho):= \rott \left( [\rho(a_1), \rho(b_1)][\rho(a_2), \rho(b_2)]...[\rho(a_g), \rho(b_g)] \right).$$
\end{definition}
Continuity of $\rott$ implies that $\euler$ is a continuous function on $\Hom(\Gamma_g, G)$ for any subgroup $G \subset \Homeo_+(S^1)$.  Furthermore, that $[a_1, b_1][a_2, b_2]...[a_g, b_g]$ is the identity in $\Gamma_g$ implies that the product of the lifted commutators
$[\tilde{\rho}(a_1), \tilde{\rho}(b_1)]...[\tilde{\rho}(a_g), \tilde{\rho}(b_g)]$ 
is an integer translation -- a lift of the identity on $S^1$ -- hence $\euler(\rho)$ is integer valued.  It follows then that $\euler(\rho)$ is constant on connected components of $\Hom(\Gamma_g, G)$.   

A remark on notation is in order.  Here, and in the sequel we use the notation $\tilde{\rho}(a)$ rather than $\widetilde{\rho(a)}$ to denote a lift of $\rho(a)$ to $\Homeo_\Z(\R)$ -- this is \emph{not} to suggest that we have lifted the representation $\rho$ to some representation $\tilde{\rho}$, but only to avoid the cumbersome notation $\widetilde{\rho(a)}$.

\boldhead{The Milnor-Wood inequality}
The \emph{Milnor-Wood inequality} implies that $\euler$ takes only finitely many values on $\Hom(\Gamma_g, G)$ for any $G \subset \Homeo_+(S^1)$. 
We recall the statement here. 
\begin{theorem}[Milnor \cite{Milnor}, Wood \cite{Wood}]
Let $\rho: \Gamma_g \to \Homeo_+(S^1)$ be a representation.  Then $| \euler(\rho) |\leq 2g-2$.  Furthermore, each integer value $n \in [-2g+2, 2g-2]$ is realized as $\euler(\rho)$ for some representation $\rho$.  
\end{theorem}
In fact, each integer $n \in [-2g+2, 2g-2]$ is realized as $\euler(\rho)$ for some representation $\rho: \Gamma_g \to \PSL(2,\R)$.  This gives a \emph{lower bound} on the number of components of $\Hom(\Gamma_g, G)$
 whenever $G$ is a subgroup of $\Homeo_+(S^1)$ containing $\PSL(2,\R)$ -- there are at least $4g-3$ connected components of $\Hom(\Gamma_g, G)$, one for each value of $\euler(\rho)$.  

\subsection{Maximal representations, cyclic covers} \label{max reps subsec}

Representations $\rho: \Gamma_g \to \PSL(2,\R)$ such that $e(\rho)$ obtains the maximal value of $\pm(2g-2)$ have a particularly nice structure, as shown by the following theorem of Goldman.

\begin{theorem}[Goldman, \cite{Goldman}] \label{fuchsian thm}
A representation $\rho: \Gamma_g \to \PSL(2,\R)$ has $\euler(\rho) = \pm(2g-2)$ if and only if it is faithful and with image a Fuchsian group.  
\end{theorem}

Using Theorem \ref{fuchsian thm} and basic properties of Fuchsian groups, one can easily show the following facts.  We leave the proofs to the reader.  

\begin{proposition} \label{maximal reps prop}
Let $\rho: \Gamma_g \to \PSL(2,\R)$ satisfy $\euler(\rho) = 2g-2$, and let $\{a_1, b_1, ... a_g, b_g\}$ be a standard generating set for $\Gamma_g$.  Then $\rho$ has the following properties.  
\begin{enumerate}[i)]
\item Every homeomorphism in the image of $\rho$ has a fixed point.
\item Each pair $(a_i, b_i)$ of generators has $\rott[\rho(a_i), \rho(b_i)] = 1$.  
In fact, there exists $x \in \R$ such that the lifted commutator satisfies $[\tilde{\rho}(a_i), \tilde{\rho}(b_i)](x) = x+1$.  
\item If $x_i \in S^1$ is any fixed point for the commutator $[\rho(a_i), \rho(b_i)]$, then the cyclic order of the points $x_i$ on $S^1$ is $x_1 < x_2 < ... < x_g$  (i.e. cyclically lexicographic).  
\end{enumerate}
\end{proposition}

An analogous result holds for representations $\rho$ with Euler number $-2g+2$.  Here the cyclic order of points will be reverse lexicographic, we will have $\rott[\rho(a_i), \rho(b_i)] = -1$ and the lifted commutator $[\tilde{\rho}(a_i), \tilde{\rho}(b_i)]$ will translate some point by $-1$.

Representations to $\PSLk$ with Euler number $\frac{2g-2}{k}$ have similar dynamical properties to those listed in Proposition \ref{maximal reps prop}.  These are described in the following proposition, which will serve as our starting point for the proof of Theorem \ref{rot rig thm}.  

\begin{proposition} \label{maximal k reps prop}
Let $k>0$ divide $2g-2$ and let $\rho: \Gamma_g \to \PSLk$ be a representation with Euler number $\frac{2g-2}{k}$.  Let $\{ a_1, b_1, ... a_g, b_g \}$ be a standard generating set for $\Gamma_g$.  The following properties hold. 
\begin{enumerate}[i)]
\item  For all $\gamma \in \Gamma_g$, there exists $m_\gamma \in \Z$ such that $\rot(\rho(\gamma)) = m_\gamma/k$ .
\item Each pair $(a_i, b_i)$ of generators has $\rot[\rho(a_i), \rho(b_i)]= 1/k$ and $\rott[\rho(a_i), \rho(b_i)] = 1/k$. 

\item Let $X_i$ be a periodic orbit for $[\rho(a_i), \rho(b_i)]$.   We may choose points $x_i^0 \in X_i$ such that if $x_i^j = [\rho(a_i), \rho(b_i)]x_i^{j-1}$, then the order of the points $x_i^j$ on $S^1$ is cyclically lexicographic of the form $$x_1^1 < x_2^1 < \ldots  < x_g^1 < x_1^2 < x_2^2 < \ldots < x_1^k <\ldots < x_g^k$$
\end{enumerate}
\end{proposition}

\begin{proof} 
Recall that $\PSLk$ is defined via the central extension $$0 \to \Z/k\Z \to \PSLk \overset{\pi}\to \PSL(2, \R) \to 1$$
and the action of $\phi \in \PSLk$ on $S^1$ is specified by
\begin{equation} \label{eq 1}
\pi \circ \phi(x) = \phi( \proj_k(x))
\end{equation}
 where $\proj_k: S^1 \to S^1$ is the $k$-fold cyclic covering map.   
In particular, for each $\gamma \in \Gamma_g$, its image $\rho(\gamma) \in \PSLk$  commutes with the order $k$ rotation of $S^1$.  

Given a representation $\rho: \Gamma_g \to \PSLk$ with $\euler(\rho) = \frac{2g-2}{k}$, consider the representation $\nu: \Gamma_g \to \PSL(2,\R)$ defined by $\nu(\gamma) = \pi \circ \rho(\gamma)$.  In other words, for each $\gamma \in \Gamma_g$, the homeomorphism $\rho(\gamma)$ is obtained from $\nu(\gamma)$ by choosing a particular lift of the action of $\nu(\gamma)$ on $S^1$ to an action on the $k$-fold cyclic cover.   By Proposition \ref{maximal reps prop},  $\nu(\gamma)$ has a fixed point, and so $\rho(\gamma)$ has a point whose orbit is contained in the orbit of an order $k$ rotation (the covering transformation of the $k$-fold cyclic cover).  It follows that $\rot(\rho(\gamma)) = m_\gamma/k$ for some integer $m_\gamma$.  

Now consider a pair of generators $(a_i, b_i)$.  We know from Proposition \ref{maximal reps prop} that if $\tilde{\nu}(a_i)$ and $\tilde{\nu}(b_i)$ are lifts of $\nu(a_i)$ and $\nu(b_i)$ to homeomorphisms of the infinite cyclic cover $\R$ of $S^1$, then there is some $x$ such that $[\tilde{\nu}(a_i), \tilde{\nu}(b_i)](x) = x+1$, i.e. $[\tilde{\nu}(a_i), \tilde{\nu}(b_i)]$ acts on $x$ by the covering transformation of $\R \to S^1$.  It follows that if we instead take the lifts $\rho(a)$ and $\rho(b)$ to the $k$-fold cyclic cover, there will be a point $x'$ in the $k$-fold cyclic cover of $S^1$ such that $[\rho(a), \rho(b)](x')$ agrees with the action of the covering transformation on $x'$, i.e. acts on $x'$ (and its orbit) by an order $k$ rotation.   

It follows that $\rot[\rho(a), \rho(b)] = 1/k$, $\rott[\rho(a), \rho(b)] = 1/k$ and that $[\tilde{\rho}(a_i), \tilde{\rho}(b_i)]$ translates any lift of $x'$ by $1/k$.   

Finally, property iii) is an immediate consequence of Property iii) of Proposition \ref{maximal reps prop} applied to $\nu$ and the fact that $\rho$ is the lift of $\nu$ to the $k$-fold cyclic cover.  

\end{proof}

As in Proposition \ref{maximal reps prop}, an analogous statement holds for $k<0$.  

\boldhead{Maximal $\PSLk$ representations}
Since representations to $\PSLk$ arise from cyclic covers, they satisfy a stronger Milnor-Wood type inequality.  Namely, if $\rho$ is any representation $\Gamma_g \to \PSLk$, then $\pi \circ \rho: \Gamma_g \to \PSL(2,\R)$ has Euler number $k\euler(\rho)$.  By the Milnor-Wood inequality, we know that $|\euler(\pi \circ \rho)| \leq 2g-2$, so $|\euler(\rho)| \leq \frac{2g-2}{k}$.  

We will refer to a representation $\rho: \Gamma_g \to \PSLk$ with $|\euler(\rho)| = \frac{2g-2}{k}$ as a \emph{maximal $\PSLk$ representation}.

\begin{remark}[\textbf{New standing assumption}] \label{positive remark}
From now on, it will be convenient to only work with representations such that $\euler(\rho) \geq 0$.  We claim we lose no generality in doing so.  Indeed, if $\{a_1, b_1, ... a_g, b_g\}$ is a standard generating set for $\Gamma_g$, and $\rho: \Gamma_g \to \Homeo_+(S^1)$ a representation such that 
$$0 > \euler(\rho) :=  \rott \left( [\rho(a_1), \rho(b_1)][\rho(a_2), \rho(b_2)]...[\rho(a_g), \rho(b_g)] \right)$$
we note that 
\begin{align*}
 0 < -\euler(\rho) & = \rott  \left( ( [\rho(a_1), \rho(b_1)][\rho(a_2), \rho(b_2)]...[\rho(a_g), \rho(b_g)])^{-1} \right)\\
& = \rott \left( ( [\rho(b_g), \rho(a_g)][\rho(b_{g-1}), \rho(a_{g-1})]...[\rho(b_1), \rho(a_1)]) \right)
\end{align*} 
In other words, if $\sigma \in \Aut(\Gamma_g)$ is defined by permuting the generators
$$\sigma(a_i) = b_{g+1-i},  \,\, \sigma(b_i) = a_{g+1-i},$$
then  $\rho \circ \sigma: \Gamma_g \to \Homeo_+(S^1)$  is a representation with positive Euler number.  Moreover, the induced map $\rho \mapsto \rho \circ \sigma$ on $\Hom(\Gamma, \Homeo_+(S^1))$ is clearly a homeomorphism, so permutes connected components.  

Thus, for the remainder of this work it will be a standing assumption that $\euler(\rho) \geq 0$ for all representations $\rho: \Gamma_g \to \Homeo_+(S^1)$, and \emph{maximal $\PSLk$ representations} are representations to $\PSLk$ with $\euler(\rho) = \frac{2g-2}{k}$.  
\end{remark}

\section{Proof strategy for Theorem \ref{rot rig thm}} \label{proof strategy sec}

\subsection{A toy case}  \label{easy case subsec}

In the special case of a maximal $\PSL^{(2)}$ representation, and with an additional assumption on the rotation numbers of some generators, we can already prove a version of \emph{rotation number rigidity}.    We do this special case now to illustrate the proof strategy of Theorem \ref{rot rig thm} and to motivate the technical work in Sections \ref{CW sec} and \ref{technical subsec}.  

\begin{proposition}[\textbf{Rotation number rigidity, toy case}]  \label{SL2 prop}
Let $\rho_0: \Gamma_g \to \SL(2, \R) = \PSL^{(2)}$ be a maximal $\PSL^{(2)}$ representation and let $S = \{a_1, b_1, ... a_g, b_g\}$ be a standard generating set for $\Gamma_g$.  Assume additionally that that $\rot(\rho_0(a_i)) = 1/2$ for all $i$.   Then $\rot(\rho(a_i))$ is constant on the component containing $\rho_0$.  
\end{proposition}

\begin{proof} 
Let $\rho_0$ be a maximal $\PSL^{(2)}$ representation as in the statement of the Proposition.  By Proposition \ref{maximal k reps prop}, $\rott[\rho_0(a_i), \rho_0(b_i)] = 1/2$ for all $i \in \{1, 2, ... ,g\}$.  
Suppose for contradiction that for some $j$, $\rot(\rho(a_j))$ is not constant on the component of $\Hom(\Gamma_g, \Homeo_+(S^1))$ containing $\rho_0$.  
By continuity of $\rot$, there exists a representation $\rho$ in this component such that 
\begin{enumerate}[i)]
\item for all $i \in \{1, 2, ... ,g\}$, we have $\rot(\rho(a_i)) \neq 0$, and
\item there exists $j$ such that $\rot(\rho(a_j))$ is either irrational or of the form $p/q$ with $q>2$.  
\end{enumerate}

It follows by Lemma \ref{comm lemma} that $\rot[\rho(a_i), \rho(b_i)] \leq 1/2$ for all $i$, and $\rot[\rho(a_j), \rho(b_j)] < 1/2$.  For simplicity, assume that $j = g$.  (This is really no loss of generality as performing a cyclic permutation of the generators will not affect our proof.)  

We estimate the Euler number of $\rho$ using the following lemma.  

\begin{lemma} \label{SL2 lemma}
Suppose that $\rho: \Gamma_g \to \Homeo_+(S^1)$ satisfies $\rott[\rho(a_i), \rho(b_i)] \leq 1/2$ for each \\ $i \in \{1, 2, \ldots ,g-1\}$.  Then 
\begin{equation} \label{eq lemma}
\rott \left( [\rho(a_1), \rho(b_1)]...[\rho(a_{g-1}), \rho(b_{g-1})]\right) \leq \frac{2g - 3}{2}.
\end{equation}
Moreover, if equality holds in \eqref{eq lemma}, then $\rott[\rho(a_i), \rho(b_i)] = 1/2$ for all $i$.  
\end{lemma}

\begin{proof}[Proof of Lemma \ref{SL2 lemma}]
Let $\rho$ satisfy $\rott[\rho(a_i), \rho(b_i)] \leq 1/2$ for all $i \in \{ 1, 2, \ldots, g-1\}$.   Inductively, assume that $\rott([a_1, b_1]\ldots [a_n, b_n] )\leq \frac{2n - 1}{2}$ with equality implying that $\rott[\rho(a_i), \rho(b_i)] = 1/2$ for all $i \in \{1, 2, ... n\}$.  The base case $n=1$ is covered by Lemma \ref{comm lemma}.  Define 
\begin{align*}
& \tilde{f} := [\tilde{a}_1, \tilde{b}_1]\ldots [\tilde{a}_n, \tilde{b}_n] \,\, \text{ and } \\  
& \tilde{g} := [\tilde{a}_{n+1}, \tilde{b}_{n+1}]
\end{align*}
and apply Theorem \ref{CW thm} to $\tilde{f}\tilde{g}$.   The conclusion of the theorem states that 
$$\rott(\tilde{f}\tilde{g}) = \rott([a_1, b_1]\ldots [a_{n+1}, b_{n+1}]) \leq \frac{2n - 1 + 1 + 1}{2} = \frac{2n + 1}{2}$$ 
with equality only if $\rott(\tilde{f}) = \frac{2n - 1}{2}$ and $\rott(\tilde{g}) = \rott([a_{n+1}, b_{n+1}]) = 1/2$, completing the inductive step.   
\end{proof}

Returning to the proof of Proposition \ref{SL2 prop} now, by definition of Euler number we have
$$[\tilde{\rho}(a_1), \tilde{\rho}(b_1)]\ldots [\tilde{\rho}(a_g), \tilde{\rho}(b_g)] = T^{\euler(\rho)}.$$
Lemma \ref{additivity lemma} implies that 
$$\rott \left( [\tilde{\rho}(a_1), \tilde{\rho}(b_1)]\ldots [\tilde{\rho}(a_{g-1}), \tilde{\rho}(b_{g-1})] \right) + \rott[\tilde{\rho}(a_g), \tilde{\rho}(b_g)] = \euler(\rho) = \frac{2g-2}{2},$$
and by Lemma \ref{SL2 lemma} we have $\rott([\tilde{\rho}(a_1), \tilde{\rho}(b_1)]\ldots [\tilde{\rho}(a_{g-1}), \tilde{\rho}(b_{g-1})]) \leq \frac{2g-3}{2}$.  It follows that $\rott[\tilde{\rho}(a_g), \tilde{\rho}(b_g)] \geq 1/2$.  But by hypothesis, $\rott[\tilde{\rho}(a_g), \tilde{\rho}(b_g)]< 1/2$, giving the desired contradiction.  

\end{proof}


\subsection{Modifications for the general case} \label{mod subsec}
Our proof of Proposition \ref{SL2 prop} made essential use of two special assumptions.  The first was that we had a standard generating set $\{a_1, b_1, ... a_g, b_g\}$ such that $\rot(\rho_0(a_i)) = 1/2$ held for all $i$.   This allowed us to conclude that a representation $\rho$ nearby to $\rho_0$ satisfied $\rott[\tilde{\rho}(a_i), \tilde{\rho}(b_i)] \leq 1/2$ for all $i$, using Lemma \ref{comm lemma}.  Put otherwise, we needed the fact that $\rho_0$ was a \emph{local maximum} for the function 
$$
\begin{array}{rl} R_i:  \Hom(\Gamma_g, \Homeo_+(S^1)) \to & \R \\
 \rho  \,\,\, \mapsto & \rott[\tilde{\rho}(a_i), \tilde{\rho}(b_i)]
\end{array}
$$
for each of the pairs $a_i, b_i$ in our standard generating set.  
(Of course, we could have reached the same conclusion under the assumption that $\rot(\rho_0(b_i)) = 1/2$, but Lemma \ref{comm lemma} does not imply anything in the case where, say, $\rot(\rho_0(a_1)) = \rot(\rho_0(b_1)) = 0$).  The meat of the proof of Theorem \ref{rot rig thm} lies in treating the case where images of generators have rotation number zero.   
This is carried out in Section \ref{technical subsec}.
(Reduction to this case is via the \emph{Euclidean algorithm for commutators of diffeomorphisms}, carried out in Section \ref{Euc alg subsec}.) 

The second key assumption in Proposition \ref{SL2 prop} was that our target group was $\PSL^{(2)}$.  We used this for the estimate on lifted rotation numbers in Lemma \ref{SL2 lemma}.  
To modify Lemma \ref{SL2 lemma} for representations to $\PSLk$, we need to bound the lifted rotation number of the product of $g-1$ homeomorphisms -- the lifted commutators $[\tilde{\rho}(a_i), \tilde{\rho}(b_i)]$ -- each with lifted rotation number $1/k$.  Unfortunately, the na\"ive approach of just repeating the argument from Lemma \ref{SL2 lemma} for these lifted commutators gives the bound
$$\rott \left( [\rho(a_1), \rho(b_1)]...[\rho(a_{g-1}), \rho(b_{g-1})]\right) \leq g-2$$
 \emph{independent of $k$}.  This is not a strong enough for our purposes; we will need to take a fundamentally different approach.

The groundwork required to solve these problems is the content of the next section.  We will undertake a detailed study of the behavior of lifted rotation numbers of products in $\Homeo_\Z(\R)$, building on the work of Calegari and Walker in \cite{CW}, and placing special emphasis on examples that arise from maximal $\PSLk$ representations.  In Section \ref{rot pf sec} we will then use these examples to first prove that $R_i$ does indeed have a local maximum at $\rho_0$ (in fact we will prove something stronger, characterizing other local maxima), and then to find a suitable replacement for Lemma \ref{SL2 lemma}.

\section{Rotation numbers of products of homeomorphisms} \label{CW sec}

In Section 3.2 of \cite{CW}, Calegari and Walker give and algorithm for computing the maximal lifted rotation number of a product $\tilde{a} \tilde{b} \in \Homeo_\Z(\R)$ given the lifted rotation numbers of $\tilde{a}$ and $\tilde{b}$. (Their algorithm also works in the more general setting of arbitrary words in $\tilde{a}$ and $\tilde{b}$ -- but not words with $\tilde{a}^{-1}$ and $\tilde{b}^{-1}$ ).    Assuming that $\rott(\tilde{a})$ and $\rot(\tilde{b})$ are rational, the algorithm takes as input the combinatorial structure of periodic orbits for $a$ and $b$ on $S^1$ (where $a$ and $b$ are the images of $\tilde{a}$ and $\tilde{b}$ under the surjection $\Homeo_\Z(\R) \to \Homeo_+(S^1))$, and as output gives the maximum possible value of $\rott(\tilde{a}\tilde{b})$, given that combinatorial structure.  

Calegari and Walker's algorithm readily generalizes to words in a larger alphabet, and we will use this generalization to prove Theorem \ref{rot rig thm}.  We find it convenient to describe the algorithm using slightly different language than that in \cite{CW}; as this will also let us treat the case (not examined in \cite{CW}) where periodic orbits of two homeomorphisms intersect nontrivially.  

\subsection{The Calegari--Walker algorithm} \label{CW subsec}

Let $c_1, c_2, \ldots, c_n$ be elements of $\Homeo_+(S^1)$ and let $\tilde{c}_1, ... \tilde{c}_n$ be lifts to $\Homeo_\Z(\R)$.  Assume that $\rott(c_i)  =  p_i/q_i$ for some integers $p_i \geq 0$ and $q_i > 0$.  

Let $X_i \subset S^1$ be a periodic orbit for $c_i$, and let $\tilde{X}_i \subset \R$ be the pre-image of $X_i$ under the projection $\R \to \R/\Z = S^1$.  If $\rot(c_i) = 0$, then one may take $X_i$ to be any finite subset of the fixed set $\fix(c_i)$.  Choose a point $x_i^0 \in \tilde{X}_i$, and enumerate the points of $\tilde{X_i}$ by labeling them $x_i^j$ in increasing order, i.e. with $x_i^j < x_i^{j+1}$.  Let $\mathbf{X}:= \{x_i^j : 1 \leq i \leq n, \, j \in \Z\} \subset \R$
We will define a dynamical system with orbit space $\mathbf{X}$ that encodes the ``maximum distance $\tilde{c_i}$ can translate points to the right."   

This system is generated by the $c_i$ acting on $\mathbf{X}$ as follows.  There is a natural left-to-right order on the points $x_i^j \in \mathbf{X}$ induced by their order as a subset of $\R$.  
Each $c_i$ acts by moving to the right, skipping over $p_i$ points of $\tilde{X}_i$ (counting the one we start on, if we start on a point of $\tilde{X}_i$) and landing on the $(p_i +1)$th.   See Example \ref{first orbit example} for an example.   We will use $c_i \cdot x_k^j$ to denote the action of $c_i$ on a point $x_k^j \in \mathbf{X}$.  Note that the action of $c_i$ is \emph{not} in any sense an action by a homeomorphism of $\mathbf{X}$.  When we want to consider $\tilde{c}_i$ as a homeomorphism, we will use the notation $\tilde{c}_i(x)$.  

Say that an orbit of this dynamical system is \emph{$\frac{l}{m}$-periodic} for a word $w$ in the alphabet $\{c_1, c_2, \ldots , c_n\}$ if there exists a point $x_i^j \in \mathbf{X}$ such that 
$$w^m \cdot x_i^j = x_i^j + l,  \, \text{ for some } l \in \Z.$$  
We call the orbit of such an $x_i^j$ an \emph{$\frac{l}{m}$-periodic orbit}.  

We claim that periodic points compute the maximum possible rotation number of the homeomorphism $w(\tilde{c}_1\ldots \tilde{c}_n)$.  Precisely, we have the following.  

\begin{proposition} \label{algorithm prop}
Suppose that $w = w(c_1, c_2, ... c_n)$ is $\frac{l}{m}$-periodic.  The following hold. 
\begin{enumerate} 
\item  $\rott(w(\tilde{c}_1,...\tilde{c}_n)) \leq \frac{l}{m}$.  
\item 

Each homeomorphism $\tilde{c}_i$ can be deformed to a homeomorphism $\tilde{d}_i \in \Homeo_\Z(\R)$ such that $\rott(w(\tilde{d}_1,...\tilde{d}_n)) = \frac{l}{m}$.  Moreover, the deformations can be carried out along a path of homeomorphisms preserving the lifted periodic orbits $\tilde{X}_i$ pointwise.   
\end{enumerate}
\end{proposition} 

\begin{proof} This follows from the work in Section 3.2 of \cite{CW}; we sketch the proof here for the convenience of the reader and to shed light on the meaning of the dynamical system in the algorithm.  
To begin, we elaborate on how the dynamical system ``computes the maximum possible rotation number".  In fact, what the action of $\tilde{c}_i$ on $\mathbf{X}$ captures is the supremal distance the \emph{homeomoprhism} $\tilde{c}_i$ (acting on $\R$) can translate a point to the right -- that $\tilde{X}_i$ is a lift of a periodic orbit implies that for any $y < x_i^j$ we have $\tilde{c}_i(y) < x_i^{j+p_i}$.  This is encoded in the dynamical system as $c_i$ ``skipping over" $p_i$ points of $\tilde{X}_i$ and landing on the $(p_i + 1)$th.    

Thus, that $w(c_1, c_2, ... c_n)^m \cdot x_i^j = x_i^j + l \in \mathbf{X}$ implies that $w(\tilde{c}_1, \tilde{c}_2, ... \tilde{c}_n) \in \Homeo_\Z(\R)$ satisfies $w(\tilde{c}_1, \tilde{c}_2, ... \tilde{c}_n)^m(x_i^j) \leq x_i^j + l$.  
If $x_i^j$ is an $\frac{l}{m}$-periodic point for $w$, we have $$w(\tilde{c}_1, \tilde{c}_2, ... \tilde{c}_n)^m(x_i^j) \leq x_i^j + l$$ and using the fact that $w(\tilde{c}_1, \tilde{c}_2, ... \tilde{c}_n)$ commutes with integer translations, this implies that 
$$w(\tilde{c}_1, \tilde{c}_2, ... \tilde{c}_n)^{km}(x_i^j) \leq x_i^j + kl$$ for all $k$.  Hence, $\rott(w(\tilde{c}_1, \tilde{c}_2, ... \tilde{c}_n)) \leq l/m$.

We will not use the second part of Proposition \ref{algorithm prop} in the sequel, so our proof sketch will be brief.  Choose a small $\epsilon > 0$ and continuously deform the action of $\tilde{c}_i$, preserving the action on $\tilde{X}_i$ to a homeomorphism $\tilde{d}_i$ that contracts each interval $(x_i^{j-1} + \epsilon, x_i^{j}]$  into $(x_i^{j+p_i} \hspace{-3pt}- \epsilon, \, x_i^{j+p_i}]$, and maps $x_i^j$ to $x_i^{j+p_i}$.  It is clear that this can be done equivariantly with respect to integer translations (i.e. through a path in $\Homeo_\Z(\R))$, and that the rotation number, which can be read off of the action on $\tilde{X}_i$, remains constant.   If $\epsilon$ is chosen close to zero, the action of $\tilde{d_i}$ on $\R$ will approximate the action of $d_i$ (and hence of $c_i$) on $\mathbf{X}$ given by the dynamical system described in our algorithm.  
 
In particular, $w(\tilde{d}_1,...\tilde{d}_n)^m(x_i^j)$ will be close to $x_i^j + l$.  The fact that $\tilde{d}_i$ contracts the interval $(x_i^j + \epsilon, x_i^{j+1}]$  into $(x_i^{j+1} \hspace{-3pt}- \epsilon, x_i^{j+1}]$ further implies that $(\tilde{d}_1...\tilde{d}_n)^{km}(x_i^j)$ approaches $x_i^j + kl$ as $k \to \infty$.  It follows from the definition of lifted rotation number that $\rott(w(\tilde{d}_1,...\tilde{d}_n)) = l/m$. 
 
\end{proof}

Let us illustrate the algorithm with an instructive example.  

\begin{example}  \label{first orbit example}
Let $c_1, c_2, c_3 \in \Homeo_+(S^1)$ satisfy $\rot(c_i)= 1/2$ for $i = 1,2,3$, and let $\tilde{c}_i$ be a lift with $\rott(\tilde{c}_i) = 1/2$.  Suppose the points $x_i^j$ have the following (lexicographic) order 
$$\ldots \, x_1^j < x_2^j < x_3^j < x_1^{j+1} < x_2^{j+1} \ldots$$
We use the algorithm to produce an upper bound for $\rott(\tilde{c}_1\tilde{c}_2\tilde{c}_3)$.  

Since $\rot(c_i) = 1/2$, we have $x_i^{j+2l} = x_i^j + l$ for all $i$ and $j$ and any $l \in \Z$.  The following diagram depicts a $5/2$-\emph{periodic orbit} of $w = c_1c_2c_3$ acting on $\mathbf{X}$.  

$$x_1^0 \overset{c_3} \longmapsto x_3^1 \overset{c_2} \longmapsto x_2^3 \overset{c_1} \longmapsto x_1^5 \overset{c_3} \longmapsto x_3^6 \overset{c_2} \longmapsto x_2^8 \overset{c_1} \longmapsto x_1^{10} = x_1^0+5$$
Hence, by Proposition \ref{algorithm prop}, $\rott(\tilde{c}_1\tilde{c}_2\tilde{c}_3) \leq 5/2$.
\end{example}

The following example is a mild generalization of Example \ref{first orbit example}. It will play a key role in Section \ref{main pf sec}.

\begin{example} \label{more general example}
Let $c_1, c_2, ... ,c_n \in \Homeo_+(S^1)$ satisfy $\rot(c_i) = 1/k$, and let $\tilde{c}_i$ be a lift of $c_i$ to $\Homeo_\Z(\R)$ satisfying $\rott(\tilde{c}_i) = 1/k$.  Suppose there are lifts $\tilde{X}_i = \{x_i^j\}$ of periodic orbits with (lexicographic) order
$$\ldots  \, x_1^j < x_2^j < \ldots  < x_n^j < x_1^{j+1}   \ldots $$  
We use the Calegari--Walker algorithm to show that $\rott(\tilde{c}_1 \tilde{c}_2 \ldots \tilde{c}_n) \leq \frac{2n-1}{k}$.

As in Example \ref{first orbit example}, since $\rot(c_i) = 1/k$, we have $x_i^{j+kl} = x_i^j + l$ for all $i$ and $j$ and any $l \in \Z$.   We compute an orbit for $c_1c_2 \ldots c_n$, starting with $x_1^0$.
\begin{equation} \label{long orbit equation}
x_1^0 \overset{c_n} \longmapsto x_n^1 \overset{c_{n-1}} \longmapsto x_{n-1}^3 \overset{c_{n-2}} \longmapsto x_{n-2}^5  \longmapsto \ldots 
 \overset{c_1} \longmapsto x_1^{2n-1} \overset{c_n} \longmapsto x_n^{2n} \overset{c_{n-1}} \longmapsto x_{n-1}^{2n+2}  \ldots 
 \end{equation}
In general, we have $c_1c_2 \ldots c_n \cdot x_1^j = x_1^{j + 2n-1}$.   After iterating this $k$ times, we have  
$$(c_1c_2 \ldots c_n)^k \cdot x_1^0 = x_1^{k(2n-1)} = x_1^0 + 2n-1$$
This gives a $\frac{2n-1}{k}$-periodic orbit, and we conclude that 
 $\rott(\tilde{c}_1 \tilde{c}_2 \ldots \tilde{c}_n) \leq \frac{2n-1}{k}$.  

\end{example}

Note that the set-up of Example \ref{more general example} mirrors the set-up of a maximal $\PSLk$ representation as described in Proposition \ref{maximal k reps prop}, with $c_i = [\rho(a_i), \rho(b_i)]$ and $\tilde{c}_i$ the lifted commutator $[\tilde{\rho}(a_i), \tilde{\rho}(b_i)]$.   

\subsection{Variations on Example \ref{more general example}} \label{variations subsec}
We now work through two variants of Example \ref{more general example} in which $\rott(\tilde{c}_1 \tilde{c}_2 \ldots \tilde{c}_n)$ fails to attain its maximal value of $\frac{2n-1}{k}$.  These will be used in the proof of Theorem \ref{rot rig thm}
when we consider \emph{deformations} of maximal $\PSLk$ representations in $\Hom(\Gamma_g, \Homeo_+(S^1))$.

\begin{example}[Deforming the combinatorial structure] \label{double points example}
Let $c_1, c_2, ... c_n \in \Homeo_+(S^1)$ satisfy $\rot(c_i) = 1/k$, and let $\tilde{c}_i$ be a lift of $c_i$ to $\Homeo_\Z(\R)$ satisfying $\rott(\tilde{c}_i) = 1/k$.  Suppose there is a lift $\tilde{X_i} = \{x_i^j\}$ of a periodic orbit for each $c_i$, ordered as
$$\ldots  \, x_1^j  \leq x_{2}^j  \leq \ldots  \leq x_n^j \leq x_1^{j+1}   \ldots$$  
with at least one instance of equality, say $x_1^0 = x_2^0$.  
We show that $\rott(\tilde{c}_1 \tilde{c}_2 \ldots \tilde{c}_n) < \frac{2n-1}{k}$ under the additional assumption that $2n-1$ and $k$ are relatively prime.  

First, note that $x_1^0 = x_2^0$ implies that $x_1^{dk} = x_2^{dk}$ for all $d \in \Z$.   
Now for \emph{any} instance where $x_i^j = x_{i+1}^j$ (call this a \emph{double point}) observe that the dynamical system exhibits the following behavior
$$ \ldots \overset{c_{i+1}} \longmapsto x_{i+1}^j = x_{i}^j \overset{c_{i}} \longmapsto x_{i}^{j+1} \overset{c_{i-1}} \longmapsto \ldots $$
whereas in the case where $x_i^j < x_{i+1}^j$ we have 
$$\ldots \overset{c_{i+1}} \longmapsto x_{i+1}^j  \overset{c_{i}} \longmapsto x_{i}^{j+2} \overset{c_{i-1}} \longmapsto \ldots $$
Informally speaking, double points ``decrease the distance that $c_i$ travels to the right."

Let $w = c_1c_2...c_n$ and consider an orbit of the action $w$ on $\mathbf{X}$.  We claim that for any initial point $x_1^j$, we have $w^k(x_1^j) \cdot = x_1^{j+s}$ for some $s < k(2n-1)$.  Indeed, the computation in Example \ref{more general example} together with the observation that ``double points decrease distance travelled to the right" shows that $s \leq k(2n-1)$, with equality only in the case where no double points are hit along the way.  In the case where no double points are hit, the orbit computation of Example \ref{more general example} holds, and we have 
$$x_1^j \overset{w} \longmapsto x_1^{j + 2n-1} \overset{w} \longmapsto x_1^{j + 2(2n-1)} \overset{w} \longmapsto ... \overset{w} \longmapsto x_i^{j+k(2n-1)}$$
However, that $2n-1$ and $k$ are relatively prime implies that $j + d(2n-1)$ is divisible by $k$ for some $d \in \{0, 1, ... k\}$, and by assumption this is a double point!  Thus, equality cannot hold in the inequality $s \leq k(2n-1)$.  

Choose $m$ such that $k$ divides $sm$.  Then $w^{km} \cdot (x_1^0) = x_1^{sm} = x_1^0 + sm/k$.  This gives 
a $\frac{sm/k}{km} = \frac{s/k}{k}$-periodic orbit for $w$, so 
$$\rott(w(\tilde{c}_1 \tilde{c}_2 \ldots \tilde{c}_n)) \leq  \frac{s/k}{k} < \frac{2n-1}{k}$$
\end{example}

\begin{remark}
The reader may find it instructive to apply the algorithm in a simple case to see what fails in example \ref{double points example} when $k$ and $n-1$ are \emph{not} relatively prime; we suggest $k = 3$, $n=4$ ,and $x_1^{3d} = x_2^{3d}$ the only instances of double points.  
\end{remark}

\boldhead{A condition for maximality} 
The Calegari--Walker algorithm produces an upper bound for the value of $\rott(w(\tilde{c}_1, ... \tilde{c}_n))$, given some data on the homeomorphisms $\tilde{c}_i$.  It is interesting to ask under what further conditions on $\tilde{c}_i$ this maximum is achieved.   As motivation, note that it is possible to have homeomorphisms $\tilde{c}_i$ satisfying the conditions in Example \ref{more general example}, but with $\rott(\tilde{c}_1 ... \tilde{c}_n) = \frac{n}{k}$ rather than the maximum value $\frac{2n-1}{k}$ given by the algorithm. Indeed, this will occur if we take each $\tilde{c}_i$ to be the translation $x \mapsto x+\frac{1}{k}$, a lift of an order $k$ rigid rotation of $S^1$.  Appropriate choice of lifted periodic orbits will satisfy the condition on the points $x_i^j$ described in the example.  

Our next proposition (\ref{max rot prop}) gives a condition for homeomorphisms $\tilde{c}_i$ as in Example \ref{more general example} to satisfy $\rott(\tilde{c}_1, ... \tilde{c}_n) = \frac{2n-1}{k}$, i.e. for the lifted rotation number to attain its maximal value.  This proposition and its corollaries will play an important role in the proof of Theorem \ref{rot rig thm}.

In what follows, when we say ``suppose $c_1, c_2, ... c_n \in \Homeo_+(S^1)$ are as in Example \ref{more general example}" we mean that $\rot(c_i) = 1/k$ for all $i$, that we have chosen lifts $\tilde{c}_i \in \Homeo_\Z(\R)$ such that $\rott(\tilde{c}_i) = 1/k$, and that we have lifts $\tilde{X_i} = \{x_i^j\}$ of periodic orbits for the $c_i$ with ordering 
$$\ldots  \, x_1^j < x_2^j < \ldots  < x_n^j < x_1^{j+1}   \ldots $$

\begin{proposition} \label{max rot prop}
Suppose $c_1, c_2, ... c_n \in \Homeo_+(S^1)$ are as in Example \ref{more general example}, and assume $k$ and $2n-1$ are relatively prime.   If $\rott(\tilde{c}_1\tilde{c}_2...\tilde{c}_n)$ attains its maximal value of $\frac{2n-1}{k}$, then for all $j \in \Z$ we have

\centering $\left\{  \begin{array}{ll}
\tilde{c}_1(x_2^j) > x_n^{j+1} \\
\tilde{c}_i(x_{i+1}^j) > x_{i-1}^{j+2} & \mbox{for } i \neq 1, n\\ 
\tilde{c}_n(x_1^j) > x_{n-1}^{j+1}
\end{array} \right.$
\end{proposition}

\begin{proof}
Note that rotation the number of a word in $\tilde{c}_i$ is invariant under cyclic permutations of the letters, meaning that $\rott(\tilde{c}_1, ... \tilde{c}_n) = \rott(\tilde{c}_n \tilde{c}_1 ... \tilde{c}_{n-1})$, etc.   The proofs of each of the inequalities above are identical, after applying a cyclic permutation and re-indexing the $x_i^j$ appropriately.   So we will give a proof of just the third inequality.  We prove the contrapositive via a straightforward computation. 

Suppose there exists $j$ such that $\tilde{c}_n(x_1^j) \leq x_{n-1}^{j+1}$.  After relabeling, we may assume that $j=0$, so $\tilde{c}_n(x_1^0) \leq x_{n-1}^{1}$.  Then
\begin{equation} \label{eq A}
\tilde{c}_{n-1} \tilde{c}_n(x_1^j) \leq \tilde{c}_{n-1}(x_{n-1}^1) = x_{n-1}^2.
\end{equation}
Considering the action of the $c_i$ on $\mathbf{X}$, we compute that 
$$(c_1 c_2 ... c_{n-2}) \cdot x_{n-1}^2 = x_1^{2n-2}$$
hence $(\tilde{c}_1 \tilde{c}_2 ... \tilde{c}_{n-2})(x_{n-1}^2) \leq x_1^{2n-2}$. Combining this inequality with \eqref{eq A} gives 
\begin{equation} \label{eq B}
(\tilde{c}_1 \tilde{c}_2 ... \tilde{c}_n)(x_{1}^0) \leq x_1^{2n-2}.
\end{equation}

Using the fact that $2n-1$ and $k$ are relatively prime, take integers $m >0$ and $l>0$ such that $m(2n-1) -1 = kl$.  
Note that $\frac{l}{m} = \frac{2n-1-1/m}{k} < \frac{2n-1}{k}$.  

The orbit computation from Example \ref{more general example} 

shows that $$(c_1 c_2 .... c_n)^{m-1}\cdot x_1^{2n-2} = x_1^{2n-2 + (m-1)(2n-1)} = x_1^{m(2n-1) -1} = x_1^{kl}$$
and so $$(\tilde{c}_1 \tilde{c}_2 ...\tilde{c}_n)^{m-1} (x_1^{2n-2}) \leq x_1^{kl}$$
Together with \eqref{eq B}, this implies that 
$$(\tilde{c}_1 \tilde{c}_2 ...\tilde{c}_n)^{m} (x_1^0) \leq x_1^{kl} = x_1^0 + l.$$
This implies that $\rott(\tilde{c}_1 \tilde{c}_2 ...\tilde{c}_n)^m \leq l$, so $\rott(\tilde{c}_1 \tilde{c}_2 ...\tilde{c}_n) \leq \frac{l}{m} < \frac{2n-1}{k}$, which is what we needed to show.  

\end{proof}

The next two corollaries of Proposition \ref{max rot prop} give strong constraints on the location of other periodic points of the homeomorphisms $c_i$.  

\begin{corollary} \label{periodic points cor}
Let $c_1, c_2, ... c_n \in \Homeo_+(S^1)$ be as in Example \ref{more general example}. Assume $k$ and $2n-1$ are relatively prime and that $\rott(\tilde{c}_1\tilde{c}_2...\tilde{c}_n)$ attains its maximal value of $\frac{2n-1}{k}$.  
Let $y_i \in S^1$ be any periodic point for $c_i$, then for any lift $\tilde{y}_i$ of $y_i$, there exists $j$ such that 

\centering $\left\{ \begin{array}{ll}  
\tilde{y}_1 \in (x_n^{j-1}, x_2^j) \\
\tilde{y}_i \in (x_{i-1}^j, x_{i+1}^j) & \mbox{for } i \neq 1, n \\
\tilde{y}_n \in (x_{n-1}^{j-1}, x_1^j)  
\end{array} \right.$
\end{corollary}

Again, the three equations are all equivalent modulo cyclic permutations.  One can interpret each as saying that \emph{any} lift $\tilde{y}_i$ of a periodic point for $c_i$ is in the same connected component of $\R \setminus ( \bigcup \limits_{m \neq i} \tilde{X}_m )$ as a lifted periodic point $x_i^j$.  

\begin{proof}
Again, we will just prove one of the inequalities, this time the second.  Let $y$ denote $\tilde{y}_i$ and suppose that $y \notin (x_{i-1}^{j-1}, x_{i+1}^j)$ for any $j$.  Then there exists $j$ such that 
$x_{i+1}^j \leq y \leq x_{i-1}^{j+1}$.

Then we have 
$$\tilde{c}_i(y) \geq \tilde{c}_i(x_{i+1}^j) > x_{i-1}^{j+2}$$ 
with the second inequality following from Proposition \ref{max rot prop}. 
Applying $(\tilde{c}_i)^{k-1}$ to both sides, we have 
$$(\tilde{c}_i)^k(y) > (\tilde{c}_i)^{k-1}(x_{i-1}^{j+2} )= x_{i-1}^{j+k+1} = x_{i-1}^{j+1} +1 \geq y+1$$
Furthermore, we know that $y \leq x_{i-1}^{j+1} < x_1^{j+1}$ and applying $(\tilde{c}_i)^k$ to both sides gives
$$(\tilde{c}_i)^k(y) < x_i^{j+1} + 1 < y+2.$$  
Thus, $1< (\tilde{c}_i)^k(y) - y < 2$, so $(\tilde{c}_i)^k(y) - y$ is not an integer and $y$ is not a lift of a periodic point for $c_i$.    

\end{proof}

\begin{corollary} \label{orbit structure cor}
Let $c_1, c_2, ... c_n \in \Homeo_+(S^1)$ be as in Example \ref{more general example}, assume $k$ and $2n-1$ are relatively prime and that $\rott(\tilde{c}_1\tilde{c}_2...\tilde{c}_n)$ attains its maximal value of $\frac{2n-1}{k}$. 
For each $i$, let $Y_i$ be \emph{any} periodic orbit of $c_i$, and let $\tilde{Y_i}$ be its lift to $\R$.  Then the points of $\tilde{Y_i}$ may be enumerated $y_i^j$ with (lexicographic) order 
$$\ldots   y_1^j < y_2^j < \ldots  < y_n^j < y_1^{j+1} <  \ldots $$  
\end{corollary} 

In other words, Corollary \ref{orbit structure cor} state that  \emph{all} choices of periodic orbits for the homeomorphisms $c_i$ have the same combinatorial structure as the original choices $X_i$.  

\begin{proof}
It suffices to prove that for any $i$ and any choice of periodic orbit $Y_i$ for $c_i$, the sets $X_1, ... X_{i-1}, Y_i, X_{i+1}, ... X_n$ have lifts that can be enumerated $\tilde{X_m^j} = \{x_m^j\}$, $\tilde{Y_i^j} = \{x_i^j\}$, with the ordering
\begin{equation} \label{x y order}
\ldots   x_1^j < x_2^j < \ldots < x_{i-1}^j < y_i^j < x_{i+1}^j < ...  < x_n^j < x_1^{j+1}  \ldots
\end{equation}

So let $Y_i$ be a periodic orbit of $c_i$ with lift $\tilde{Y}_i \subset \R$, and let $y \in \tilde{Y_i}$.  By Corollary \ref{periodic points cor}, there exists $j$ such that $x_{i-1}^j < y < x_{i+1}^j$.    Let $y = y_i^j$.  We claim that $\tilde{c}_i(y_i^j) \in (x_{i-1}^{j+1}, x_{i+1}^{j+1})$. 
Indeed, we know that $\tilde{c}_i(y_i^j) \in (x_{i-1}^{m}, x_{i+1}^{m+1})$ for some $m$, and using the inequality above, we have 
$$x_i^j = \tilde{c}_i(x_i^{j-1}) < \tilde{c}_i(y_i^j) < \tilde{c}_i(x_{i+1}^j) < \tilde{c}_i(x_i^{j+1}) = x_i^{j+2}$$   
so it follows that $m = j-1$.  
The same argument now shows that $(\tilde{c}_i)^d(y_i^j) \in (x_{i-1}^{j+d}, x_{i+1}^{j+d})$ for all $d \in \Z$.  Defining $y_i^{j+d} := (\tilde{c}_i)^d(y_i^j)$, we have a complete enumeration of the points of $\tilde{Y}_i$ satisfying the condition \eqref{x y order}.  

\end{proof}

\begin{remark} We conclude this section by noting (for the experts) that Proposition \ref{max rot prop} should really be interpreted as a more general version of Matsumoto's Theorem 2.2 in \cite{Matsumoto} regarding ``tame" elements of $\Homeo_\Z(\R)$.  
\end{remark} 

\subsection{Homeomorphisms with fixed points}  \label{fixed points subsec}
When describing the Calegari--Walker algorithm in Section \ref{CW subsec}, we mentioned that if $\rot(\tilde{c}_i) = 0$, then one can use \emph{any} finite subset $X_i$ of $\fix(c_i)$ as input for the algorithm.   The following example shows that the resulting bound on $\rott(w(\tilde{c}_1, ... \tilde{c}_n))$ may depend on the choice of the $X_i$.  
This, and the examples that follow, will also play a role in the proof of Theorem \ref{rot rig thm}.  

\begin{example} \label{different estimates ex} 
Let $c_1$ and $c_2$ be elements of $\Homeo_+(S^1)$ with $\rot(c_1) = \rot(c_2) = 0$.  Choose lifts $\tilde{c}_i$ of $c_i$ such that $\rott(\tilde{c}_i) = 0$.    Suppose there exist subsets $X_1 \subset \fix(c_1)$ and $X_2 \subset \fix(c_2)$ each of cardinality $k$, and such that every two points of $X_1$ are separated by a point of $X_2$ (as in Figure \ref{alternating pairs fig}).    We call such an arrangement of fixed points \emph{k alternating pairs}.  

 \begin{figure*}
   \labellist 
  \small\hair 2pt
   \pinlabel $\fix(c_1)$ at 150 150 
   \pinlabel $\fix(c_2)$ at 180 040
   \endlabellist
  \centerline{
    \mbox{\includegraphics[width=1.1in]{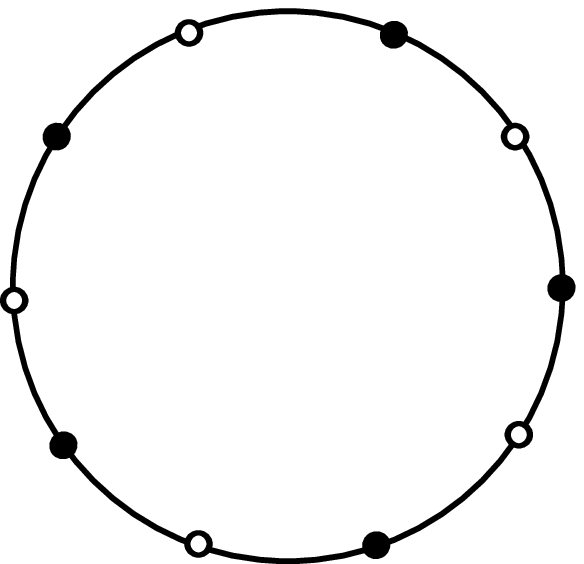}}}
 \caption{ \small $k$ ``alternating pairs" of fixed points for $c_1$ and $c_2$}
\vspace{-2pt} \centerline{\small \mbox{Dark circles are fixed points of $c_1$, white circles are fixed points of $c_2$}}
  \label{alternating pairs fig}
  \end{figure*}
  
Then the lifts $\tilde{X_1}$ and $\tilde{X_2}$ can be ordered 
$$ ... x_1^0 < x_2^0 < x_1^1 < x_2^1 <x_1^2 < x_2^2 ... $$
and we have a $1/k$-periodic orbit 
$$ x_1^0 \overset{c_2} \longmapsto < x_2^0 \overset{c_1} \longmapsto < x_1^1 \overset{c_2} \longmapsto x_2^1\overset{c_1} \longmapsto ...  \overset{c_2} \longmapsto x_2^{k-1} \overset{c_1} \longmapsto x_1^{k}$$
Hence, $\rott(\tilde{c}_1 \tilde{c}_2) \leq 1/k$.  

Now we work through the algorithm with different choices of subsets of $\fix(c_1)$ and $\fix(c_2)$ as input.  Suppose, for instance, that we choose subsets $X_i' \subset \fix(c_i)$  that are singletons.  Then the lifts of points $x_i^j \in \tilde{X}_i'$ are again ordered 
$$ ... x_1^0 < x_2^0 < x_1^1 < x_2^1 ... $$
but now $x_i^{j+1} = x_i^j +1$, and 
$$ x_1^0 \overset{c_2} \longmapsto  x_2^0 \overset{c_1} \longmapsto  x_1^1$$ is a 1-periodic orbit.  This gives the (weaker) bound $\rott(\tilde{c}_1 \tilde{c}_2) \leq 1$.  
\end{example}

The computation in Example \ref{different estimates ex} above illustrates the following general phenomenon.  

\begin{proposition}  \label{different estimates prop} 
Let $X_i' \subset X_i$ be finite subsets of $\fix(c_i)$, and let $w$ be any word in $\tilde{c}_1, ... \tilde{c}_n$.   Let $r$ be the estimate $\rott(w) \leq r$ produced by the Calegari--Walker algorithm with inputs $X_i$, and let $r'$ be the estimate $\rott(w) \leq r'$ produced by the algorithm with input $X_i'$.  Then $r' \leq r$.  
\end{proposition}

The proof is elementary and we leave it to the reader, as we will not use this proposition in the sequel.  However, we do make one instructive remark which will come into play later.    

\begin{remark} \label{doubled points rk}
Note that the inequality in Proposition \ref{different estimates prop} may not be strict -- in other words, putting more information (a larger fixed set) into the Calegari--Walker algorithm does not always give a better estimate.  For example, suppose we modify Example \ref{different estimates ex} by giving each of the $k$ points of $X_1$ and $X_2$ a nearby ``double," as in Figure \ref{alternating doubles fig}, to create sets $X_1'$ and $X_2'$ each of cardinality $2k$.

The lifts $\tilde{X}_i' = \{x_i^j\}$ -- excuse the  abuse of notation-- are now ordered
$$ ... x_1^0 < x_1^1 < x_2^0 < x_2^1 < x_1^2 < x_1^3 < x_2^2 < x_2^3 ... $$
with $x_1^{j+2k} = x_1^{j} + 1$.  Then
$$ x_1^0 \overset{c_2} \longmapsto < x_2^0 \overset{c_1} \longmapsto < x_1^2 \overset{c_2} \longmapsto x_2^2 ...  \overset{c_2} \longmapsto x_2^{2k-2} \overset{c_1} \longmapsto x_1^{2k}$$ 
is again a $1/k$-periodic orbit, giving the estimate $\rott(\tilde{c}_1 \tilde{c}_2) \leq 1/k$.  
\end{remark} 

 \begin{figure*}
   \labellist 
  \small\hair 2pt
   \pinlabel $\fix(c_1)$ at 185 90 
   \pinlabel $\fix(c_2)$ at 175 040
   \endlabellist
  \centerline{
    \mbox{\includegraphics[width=1.3in]{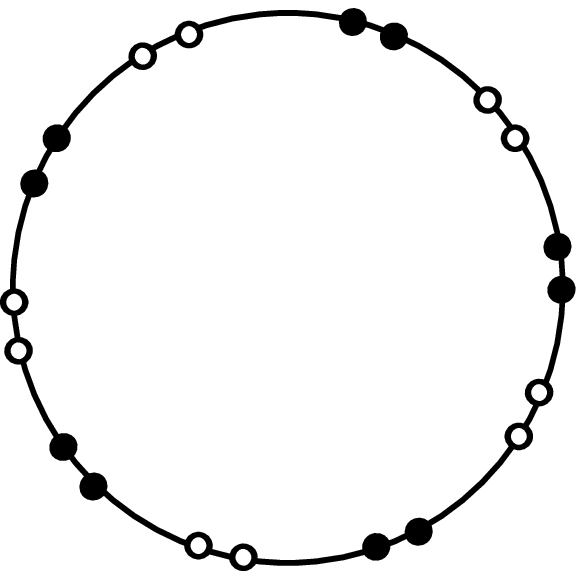}}}
    \vspace{-3pt}
 \caption{ \small $k$ \emph{doubled} alternating pairs of fixed points.}
\vspace{-2pt} \centerline{\small \mbox{Dark circles are fixed points of $c_1$, white circles are fixed points of $c_2$}}
  \label{alternating doubles fig}
  \end{figure*}

With these examples as tools, we move on to the proof of Theorem \ref{rot rig thm}. 

\section{Proof of Theorem \ref{rot rig thm}} \label{rot pf sec}
Fix a standard generating set $S = \{a_1, b_1, ... a_g, b_g\}$ for $\Gamma_g$.   We will use this generating set until the end of Section \ref{Euc alg subsec}.

Define $R_i: \Hom(\Gamma_g, \Homeo_+(S^1)) \to \R$ by 
$$R_i(\rho) = \rott[\tilde{\rho}(a_i), \tilde{\rho}(b_i)].$$  
Recall from our proof outline in Section \ref{mod subsec} that our first goal will be to identify show that maximal $\PSLk$ representations are local maxima for the functions $R_i$.  In other words, for any maximal $\PSLk$ representation $\rho_0$, there is an open neighborhood $U_i(\rho_0) \subset \Hom(\Gamma_g, \Homeo_+(S^1))$ such that $R_i(\rho) \leq R_i(\rho_0)$ holds for all $\rho \in U_i(\rho_0)$.   Our strategy -- which will also give us extra information that will be critical in later steps of the proof -- is to define sets $N_i$ of representations that share some characteristics of maximal $\PSLk$ representations.  In particular, our work in Section \ref{CW sec} will imply that $R_i(\rho) \leq R_i(\rho_0)$ for all $\rho \in N_i$ and any maximal $\PSLk$ representation $\rho_0$.  We will then identify \emph{interior points} of $N_i$ and show in particular that maximal $\PSLk$ representations are interior points.  

\subsection{Identifying local maxima of $R_i$} \label{technical subsec}
Let $\rho_0$ be a maximal $\PSLk$ representation.  We will work first under the assumption that $\rot(\rho_0(a_i)) =0$ holds for all $i$.  Then $\rot(\rho_0(a_i)^{-1}) = \rot(\rho_0(b a^{-1} b^{-1})) = 0$ as well.  The reduction to the case that $\rot(\rho_0(a_i)) =0$ is carried out in Section \ref{Euc alg subsec}.  

Fix $i \in \{1, 2, ... g\}$.  It will be convenient to set up the following notation, consistent with the notation in Section \ref{CW sec}.

\begin{notation} 
For $\rho \in \Hom(\Gamma_g, \Homeo_+(S^1))$ define 
$$c_1(\rho) := \rho(a_i),  \, \text{ and } \, c_2(\rho) := \rho(b_i a_i^{-1} b_i^{-1}).$$ 
\end{notation} 

\noindent Note that, for any choice of lifts $\tilde{\rho}(a_i)$ and $\tilde{\rho}(b_i)$ of $\rho(a_i)$ and $\rho(b_i)$ to $\Homeo_\Z(\R)$, the lifts 
 $$\tilde{c}_1(\rho) := \tilde{\rho}(a_i), \text{ and } \tilde{c}_2(\rho) := \tilde{\rho}(b_i) (\tilde{\rho}(a_i))^{-1} (\tilde{\rho}(b_i))^{-1}$$
 of $c_1(\rho)$ and $c_2(\rho)$ satisfy $\rott(\tilde{c}_1(\rho)\tilde{c}_2(\rho)) = R_i(\rho)$.   
 
Our goal is to use the work of Section \ref{CW sec} to give bounds for $\rott(\tilde{c}_1(\rho)\tilde{c}_2(\rho))$, given certain combinatorial data.  
We motivate this by examining the combinatorial structure of $\fix(c_1(\rho_0))$ and $\fix(c_2(\rho_0))$, where $\rho_0$ is our maximal $\PSLk$ representation.  

\boldhead{Structure of $\fix(c_1(\rho_0))$ and $\fix(c_2(\rho_0))$.}
Let $X_1 = \fix(c_1(\rho_0))$ and $X_2 = \fix(c_2(\rho_0))$.    We claim these sets are ordered in $S^1$ exactly as the ``doubled alternating pairs" in Figure \ref{alternating doubles fig}.   To see this, consider first the action of a pair of standard generators $(a_i, b_i)$ under an injective, Fuchsian representation $\nu: \Gamma_g \to \PSL(2,\R)$ with Euler number $2g-2$.  This is depicted in Figure \ref{fuchsian fig}.  

 \begin{figure*}
   \labellist 
  \small\hair 2pt
   \pinlabel $\nu(a_i)$ at 115 53 
   \pinlabel $\nu(b_i)$ at 50 115
   \pinlabel $\nu(b_ia_i^{-1}{b_i}^{-1})$ at 122 88
   \endlabellist
  \centerline{
    \mbox{\includegraphics[width=1.5in]{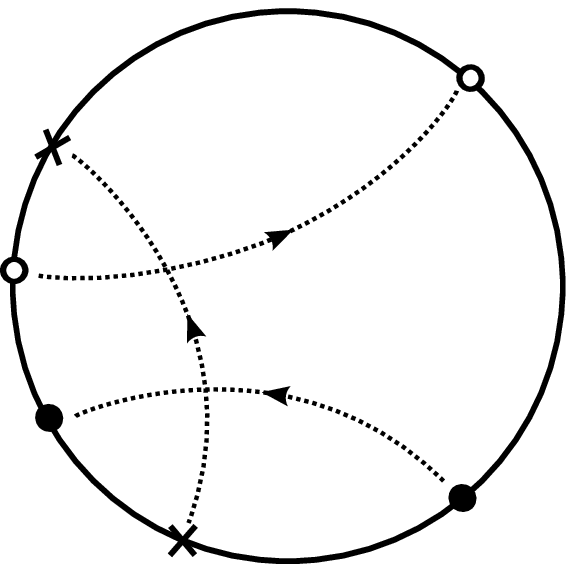}}}
 \caption{\small Dynamics of the images of standard generators under a Fuchsian representation $\nu$}
\vspace{-2pt} \centerline{\small \mbox{Dotted arrows indicate the axes of hyperbolic elements acting on the disc}}
  \label{fuchsian fig}
  \end{figure*}
\vspace{2pt} 

Our maximal $\PSLk$ representation $\rho_0$ is a lift of such a representation $\nu$ to the $k$-fold cover of $S^1$, with $X_1$ and $X_2$ the lifts of $\fix(\nu(a_i))$ and $\fix(\nu(b_i a_i^{-1} b_i ^{-1}) )$, respectively.  The fact that the points of
$\fix(\nu(a_i))$ do not separate those of $\fix(\nu(b_i a_i^{-1} b_i ^{-1} ))$ on $S^1$ implies that points of $X_1$ and $X_2$ on the $k$-fold cover appear in ``doubled alternating pairs" exactly as in Figure \ref{alternating doubles fig} from Remark \ref{doubled points rk}.  

As in Example \ref{different estimates ex}, the lifted fixed points $\tilde{X_1} = \{x_1^j\}$ and $\tilde{X_2} = \{x_2^j\}$ may be ordered
$$... x_1^0 < x_1^1 < x_2^0 < x_2^1 < x_1^2 < x_1^3 < x_2^2 < x_2^3 ... $$
We also know that $\tilde{\rho_0}(b_i)(\tilde{X}_1) = \tilde{X}_2$.  Since $\rott(\tilde{\rho_0}(b_i)) = \frac{n}{k}$ for some $n$, it follows from the dynamics of the action of $\nu(b_i)$ in the Fuchsian case (as in Figure \ref{fuchsian fig}) that 
$\tilde{\rho}_0(b_i)(x_1^{2j}) = x_2^{2j+2n-1}$ and $\tilde{\rho}_0(b_i)(x_1^{2j+1}) = x_2^{2j+2n}$ for all $j$.  See Figure \ref{containment fig}.

\boldhead{Representations with good fixed sets.}
For general $\rho \in \Hom(\Gamma_g, \Homeo_+(S^1))$, we say that $c_1(\rho)$ has a \emph{good fixed set} if it fixes a set of points that ``look like" the set $X_1(\rho_0)$ above.

 \begin{figure*}
   \labellist 
  \small\hair 2pt
   \pinlabel $x_1^{2n-1}$ at 30 25 
   \pinlabel $x_2^{2n-2}$ at  85 25
    \pinlabel {$x_2^{2n-1}\hspace{-4pt}=\hspace{-2pt} \tilde{\rho}_0(b_i)x_1^0$} at 160 25
    \pinlabel $x_1^{2n}$ at 230 25
    \pinlabel $x_1^{2n+1}$ at 290 25 
   \pinlabel {$x_2^{2n}\hspace{-2pt}=\hspace{-2pt} \tilde{\rho}_0(b_i)x_1^{1}$} at 360 25
    \pinlabel $x_2^{2n+1}$ at 430 25 
     \pinlabel $x_1^{2n+2}$ at 490 25 
   \endlabellist
  \centerline{
    \mbox{\includegraphics[width=5.3in]{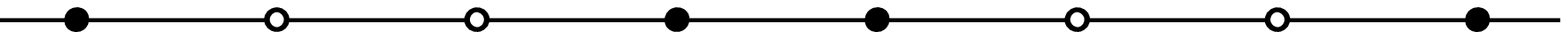}}}
 \caption{\small Order of points $x_i^j$ and $\tilde{\rho}_0(b_i)x_i^j$}
  \label{containment fig}
  \end{figure*}

\begin{definition} 
For $\rho \in \Hom(\Gamma_g, \Homeo_+(S^1))$, say that $X_1(\rho) \subset \fix(c_1(\rho))$ is a \emph{good fixed set} for $c_1(\rho) = \rho(a_i)$ if the following hold: 
\begin{enumerate}[i)]
\item $X_1(\rho)$ is of cardinality $2k$ 
\item If we define $X_2(\rho)$ by $X_2(\rho) := \rho(b_i)(X_1(\rho)) \subset \fix(\rho(b_i a_i^{-1} b_i^{-1})) = \fix(c_2(\rho))$, then the sets of lifts $\tilde{X_1}(\rho) = \{x_1^j(\rho)\}$ and $\tilde{X_2}(\rho) = \{x_2^j(\rho)\}$ can be ordered 
$$... x_1^0(\rho) < x_1^1(\rho) < x_2^0(\rho) < x_2^1(\rho) < x_1^2(\rho) < x_1^3(\rho) < x_2^2(\rho) < x_2^3(\rho) ... $$
\item With the ordering above, we have 
\begin{align*}
\tilde{\rho}(b_i)( x_1^{2j}(\rho) ) = x_2^{2j+2n-1}(\rho)  \mbox{  and   }\,
\tilde{\rho}(b_i)( x_1^{2j+1}(\rho) ) = x_2^{2j+2n}(\rho)
\end{align*}
\end{enumerate} 
\end{definition}
\noindent Example \ref{different estimates ex} shows that if $\rho(a_i)$ has a good fixed set, then $\rott[\tilde{\rho}(a_i), \tilde{\rho}(b_i)] = R_i(\rho) \leq 1/k$.   

Let $N_i$ be the closure in $\Hom(\Gamma_g, \Homeo_+(S^1)$ of the set $$\{ \rho \in \Hom(\Gamma_g, \Homeo_+(S^1)) \mid \rho(a_i) \text{ has a good fixed set }\}.$$  
Then $R_i(\rho) \leq 1/k$ holds for all $\rho \in N_i$ as well.  We claim that representations in $N_i$ where $R_i(\rho)$ attains the maximum value are \emph{interior points} of $N_i$.  

\begin{proposition}[\textbf{Interior points of $N_i$}] \label{int Ni prop}
Let $\rho \in N_i$ satisfy $R_i(\rho)  = 1/k$.  Then $N_i$ contains an open neighborhood of $\rho$ in $\Hom(\Gamma_g, \Homeo_+(S^1))$.  
\end{proposition}

\begin{proof} 

Take $\rho \in N_i$ such that $R_i(\rho)= 1/k$.  Since $\rho \in N_i$, we may take a sequence of representations $\rho_n$ approaching $\rho$  such that each $\rho_n$ has a good fixed set $X_1(\rho_n) \subset \fix(c_1(\rho_n))$.   Since the space of $k$-tuples in $S^1$ is compact, after passing to a subsequence, we may assume that $X_1(\rho_n)$ converges to a $2k$-tuple of (a priori not necessarily distinct) points $X_1(\rho) \subset S^1$. We claim that 
\begin{enumerate}[i)]
\item $X_1(\rho)$ is a good fixed set for $c_1(\rho)$, and
\item  $X_1(\rho)$ contains ``stable fixed points", in a sense (to be made precise below) that will imply that all representations sufficiently close to $\rho$ will also have good fixed sets, and so lie in $N_i$.
\end{enumerate}

\boldhead{i) Proof that $X_1(\rho)$ is a good fixed set} 
That $\rho_n$ converges to $\rho$ implies that $X_1(\rho) \subset \fix(c_1(\rho))$.  
Let $\tilde{X}_1(\rho_n) \subset \R$ be the set of all lifts of points in $X_1(\rho_n)$. 
If we choose to enumerate $X_1(\rho_n)$ so that the points $x_1^0(\rho_n)$ stay within a compact set, then after a further subsequence we may assume that for each $j$, the sequence $x_1^j(\rho_n)$ converges to a point $x_1^j(\rho) \in \tilde{X_1}(\rho)$, and so $x_2^j(\rho_n)$ converges to some $x_2^j(\rho) \in \tilde{X}_2(\rho)$ as well.  These points are then ordered
\begin{equation} \label{leq order eq}
...  x_1^0(\rho) \leq x_1^1(\rho) \leq x_2^0(\rho) \leq x_2^1(\rho) \leq x_1^2(\rho) \leq x_1^3(\rho) \leq x_2^2(\rho) \leq x_2^3(\rho) ... 
\end{equation}
with
\begin{equation} \label{containment eq}
\tilde{\rho}(b_i)( x_1^{2j}(\rho) ) = x_2^{2j+2n-1}(\rho),  \mbox{  and  }\,
\tilde{\rho}(b_i)( x_1^{2j+1}(\rho) ) = x_2^{2j+2n}(\rho)
\end{equation}

Thus, we need only show that all the inequalities in \eqref{leq order eq} are strict.  This is straightforward: if $x_1^j(\rho) = x_2^l(\rho)$ for some pair $l$ and $j$, then $\fix(c_1(\rho)) \cap \fix(c_2(\rho)) \neq \emptyset$.  Thus, $\rott(\tilde{c}_1(T) (\tilde{c}_2(T)) = 0$, contradicting our assumption.  Since $x_1^l(\rho) \neq x_2^j(\rho)$ for all $l, j$, it follows from Equation \eqref{containment eq} that $x_1^j(\rho) \neq x_1^l(\rho)$ for all $l \neq j$ as well.

\boldhead{ii) Stability phenomena} 
We now deal with point ii), namely, showing that all representations sufficiently close to $\rho$ will also have good fixed sets.     
To do this, we study the dynamics of $\tilde{c}_1(\rho)$ and $\tilde{c}_2(\rho)$ and identify some stable behavior for fixed points.  
To make the next part of argument easier to read, we drop ``$\rho$" from the notation, writing $x_i^j$ for $x_i^j(\rho)$.  

We start with two dynamical lemmas. 
\begin{lemma} \label{monotone lemma 1} 
$\tilde{c}_2(\rho)$ is monotone increasing on each interval $[x_1^{2j}, x_1^{2j+1}]$.  
\end{lemma}

\begin{proof}  
We use the fact that $R_i(\rho) = \rott(\tilde{c}_1(\rho) \tilde{c}_2(\rho)) = \frac{1}{k}$.  

Suppose for contradiction that $\tilde{c}_2(\rho)$ is not monotone increasing on some $[x_1^{2j}, x_1^{2j+1}]$.  There are three cases to consider.  First, suppose that $\tilde{c}_2(\rho)$ has a fixed point $y$ in $[x_1^{2j}, x_1^{2j+1}]$ for some $j$.  Then let $\bar{y} \in S^1$ be the projection of $y$ to $S^1 = \R/\Z$ and consider the sets 
$$Y_2 = X_2(\rho) \cup \{y\} \subset \fix(c_2(\rho)), \text{ and}$$ 
$$Y_1 = \rho(b_i)^{-1}(Y_2) \subset \fix(c_1(\rho)).$$  Now $Y_1$ and $Y_2$ contain $k+1$ \emph{alternating pairs}, so by Example \ref{different estimates ex} $\rott(\tilde{c}_1(\rho) \tilde{c}_2(\rho)) \leq \frac{1}{k+1}$, a contradiction.  

If instead $\tilde{c}_2(\rho)$ fixes an endpoint of $[x_1^{2j}, x_1^{2j+1}]$, then $\tilde{c}_1(\rho)$ and $\tilde{c}_2(\rho)$ have a common fixed point, and so $\rott(\tilde{c}_1(\rho) \tilde{c}_2(\rho)) = 0$.   Finally, if $\tilde{c}_2(\rho)$ is monotone \emph{decreasing} on $[x_1^{2j}, x_1^{2j+1}]$, then $\tilde{c}_1(\rho)\tilde{c}_2(\rho)(x_1^{2j+1}) < \tilde{c}_1(\rho)(x_1^{2j+1}) = x_1^{2j+1}$ and so $\rott(\tilde{c}_1(\rho) \tilde{c}_2(\rho)) \leq 0$, again a contradiction.  

\end{proof}

Essentially the same argument can be used to prove the following lemma (we omit the proof).  
\begin{lemma}  \label{monotone lemma 2} 
$\tilde{c}_1(\rho)$ is monotone increasing on each interval $[x_2^{2j}, x_2^{2j+1}]$. 
\end{lemma} 

\noindent Now since $\tilde{c}_2(\rho) = \tilde{\rho}(b_i)^{-1} \tilde{\rho}(a_i)^{-1} \tilde{\rho}(b_i)$ is monotone increasing on $[x_1^{2j + 2n}, x_1^{2j + 2n +1}]$, its inverse $\tilde{\rho}(b_i)^{-1} \tilde{\rho}(a_i) \tilde{\rho}(b_i)$ is monotone \emph{decreasing} on $[x_1^{2j +2n }, x_1^{2j+2n+1}]$, and so 
 $\tilde{c}_1(\rho) = \tilde{\rho}(a_i)$ is monotone decreasing on 
 $$\tilde{\rho}(b_i)^{-1}[x_1^{2j +2n }, x_1^{2j + 2n+1}] \subset (x_1^{2j}, x_1^{2j+1}).$$
The dynamics of  $\tilde{c}_1(\rho)$ in this case are shown in Figure \ref{dynamics of c1 fig}, with arrows indicating where $\tilde{c}_1(\rho)$ is increasing and decreasing.  
 
 \begin{figure*}
   \labellist 
  \small\hair 2pt
   \pinlabel $x_1^{2j-1}$ at 30 -8 
   \pinlabel $x_2^{2j-2}$ at  85 -8 
    \pinlabel $x_2^{2j-1}$ at 145 -8
    \pinlabel $x_1^{2j}$ at 205 -8
    \pinlabel $x_1^{2j+1}$ at 305 -8 
   \pinlabel $x_2^{2j}$ at 370 -8
    \pinlabel $x_2^{2j+1}$ at 430 -8 
     \pinlabel $x_1^{2j+2}$ at 500 -8 
   \endlabellist
  \centerline{
    \mbox{\includegraphics[width=4.5in]{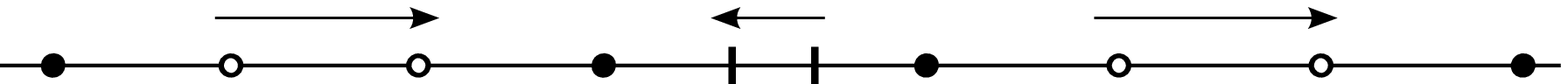}}}
 \caption{\small dynamics of $\tilde{c}_1(T)$}
  \label{dynamics of c1 fig}
  \end{figure*}

 It follows from this dynamical picture that 
  \begin{align*}
& \fix(c_1(\rho)) \cap \left(x_2^{2j-1},\,\, \tilde{\rho}(b_i)^{-1}(x_1^{2j+2n} \right) \neq \emptyset, \text{ and } \\
& \fix(c_1(\rho)) \cap \left(\tilde{\rho}(b_i)^{-1}(x_1^{2j+2n+1}), \,\, x_2^{2j} \right) \neq \emptyset,
\end{align*}
and our next step is to show that these intervals contain ``attracting" and ``repelling" fixed sets, which will satisfy a stability property.  

\begin{definition}
Let $f \in \Homeo_+(\R)$.  A connected component $Z$ of $\fix(f)$ is called \emph{attracting} for $f$ if for any sufficiently small neighborhood $Z'$ of $Z$ we have $\bigcap \limits_{n>0} f^n(Z') = Z$.  Similarly, a connected component $Z \subset \fix(f)$ is called \emph{repelling} for $f$ if for any sufficiently small neighborhood $Z'$ of $Z$ we have $\bigcap \limits_{n>0} f^{-n}(Z') = Z$. 
\end{definition}

\begin{lemma}[\textbf{how to find attracting and repelling sets}] \label{attracting lemma}
Let $f \in \Homeo_+(\R)$, and let $x<y \in \R$.  If $f(x) > x$ and $f(y) < y$, then $\fix(f) \cap (x,y)$ contains an attracting set.  Similarly, if $f(x) < x$ and $f(y) > y$, then $\fix(f) \cap (x,y)$ contains a repelling set.
\end{lemma}
The proof of Lemma \ref{attracting lemma} is elementary and we leave it as an exercise.   It is also easy to see that attracting and repelling fixed sets contain \emph{stable fixed points} in the following sense.

\begin{lemma}[\textbf{fixed point stability}] \label{stable fix points lemma}
Let $f \in \Homeo_+(\R)$, and let $Z$ be an attracting or a repelling fixed set for $f$.  Then for any sufficiently small neighborhood $Z'$ of $Z$, there exists $\epsilon >0$ (depending on $f$ and $Z'$) such that if $g \in \Homeo_+(\R)$ with $|f-g|<\epsilon$ (using the uniform norm), then $\fix(g) \cap Z' \neq \emptyset$.  
\end{lemma}
Again, we leave the proof as an exercise.  
\bigskip

Now to return to the situation at hand. 
The dynamics of $\tilde{c}_1(\rho)$ as shown in Figure \ref{dynamics of c1 fig}, together with Lemma \ref{attracting lemma}, implies that for each $j \in \Z$ there exists an attracting set 

\begin{equation} \label{first Z eq}
Z^{2j} \subset \left(x_2^{2j-1}, \,\, \tilde{\rho}(b_i)^{-1}T^{-1}(x_1^{2j+2n}\right) \cap \fix(\tilde{c}_1(\rho))
\end{equation} 
and a repelling  set 
\begin{equation} \label{second Z eq}
Z^{2j+1} \subset \left( \tilde{\rho}(b_i)^{-1}(x_1^{2j+2n+1}),\,\, x_2^{2j} \right) \cap \fix(\tilde{c}_1(\rho)).
\end{equation}

Since $\tilde{c}_1(\rho)$ and $\tilde{c}_2(\rho)$ commute with integer translations, and since $x_i^{m+2k} = x_i^{m} + 1$ holds for all $m$, we may in fact choose attracting and repelling sets such that each set $Z^{m+2k}$ is the image of $Z^{m}$ under translation by 1.


In the next lemma, we use the notation $P < Q$ for sets $P$ and $Q$ to mean that for all points $p \in P$ and $q \in Q$ we have $p<q$.  Similarly, if $x$ is a point and $P$ a set, ``$x<P$" means that $\{x\} < P$.  

\begin{claim} \label{Z claim}
If $\rott(\tilde{c}_1\tilde{c}_2) \geq 1/k$, then for all $j$ we have 
\begin{align} 
& Z^{2j + 2n +1} < \tilde{\rho}(b_i)(Z^{2j + 1}) <  Z^{2j + 2n + 2} \text{ and } \label{zclaim1} \\
& Z^{2j + 2n -1} < \tilde{\rho}(b_i)(Z^{2j}) <  Z^{2j + 2n} \label{zclaim2}
\end{align}
\end{claim}

Note the similarity between the equations in Claim \ref{Z claim} and and the ordering of points in a ``good fixed set", which by definition is
\begin{align*}
& x_1^{2j + 2n +1} < x_2^{2j+2n+1} = \tilde{\rho}(b_i)(x_1^{2j + 1}) <  x_1^{2j + 2n + 2} \text{ and } \\
& x_1^{2j + 2n -1} < x_2^{2j+2n-1} = \tilde{\rho}(b_i)(x_1^{2j}) < x_1^{2j + 2n}.
\end{align*}

\begin{proof}[Proof of Claim \ref{Z claim}]
Assume that $\rott(\tilde{c}_1\tilde{c}_2) \geq 1/k$. 
One side of each of the inequalities in Claim \ref{Z claim} will follow easily from Equations \eqref{first Z eq} and \eqref{second Z eq}.  These two equations are equivalent to 
 \begin{align*}
& x_2^{2j-1} < Z^{2j} < \tilde{\rho}(b_i)^{-1}(x_1^{2j+2n}) \,  \mbox{ and } \\
 & \tilde{\rho}(b_i)^{-1}(x_1^{2j+2n+1}) < Z^{2j+1} < x_2^{2j}, 
\end{align*}
which imply that 
\begin{align*}
& \tilde{\rho}(b_i)(Z^{2j}) < x_1^{2j+2n} < x_2^{2j+2n+1} < Z^{2j + 2n + 2} \,  \mbox{ and } \\
& Z^{2j + 2n +1}< x_2^{2j + 2n -2} < x_1^{2j+2n+1} < \tilde{\rho}(b_i)(Z^{2j+1})
\end{align*}
giving the right side of \eqref{zclaim1} and left side of \eqref{zclaim2}.  

To show the left side of \eqref{zclaim1}, suppose for contradiction that for some $j$ there are points $z^{2j+1} \in Z^{2j+1}$ and $z^{2j+2n+1} \in Z^{2j+2n+1}$ such that $\tilde{\rho}(b_i)(z^{2j+1}) \leq z^{2j + 2n +1}$.   If equality holds, then $\tilde{c}_1(\rho)$ and $\tilde{c}_2(\rho)$ have a common fixed point and $\rott(\tilde{c}_1(\rho)\tilde{c}_2(\rho))=0$, contradicting our assumption.  
If instead we have strict inequality, then let $\bar{z}$ denote the projection of $z^{2j+1}$ to $S^1$ and let $Y_1 = (X_1(\rho) \cup \{\bar{z}\}) \subset  \fix(c_1(\rho))$, and $Y_2 =  \rho(b_i)(Y_1) \subset \fix(c_2(\rho))$.  Since 
$$x_1^{2j+2n+1} < \tilde{\rho}(b_i)(z^{2j+1}) < z^{2j + 2n +1} < x_2^{2j+2n + 1}$$ 
the sets $Y_1$ and $Y_2$ contain $k+1$ \emph{alternating pairs} (see Figure \ref{zj figure}, indexed with $j=0$).
It then follows from Example \ref{different estimates ex} that $\rott(\tilde{c}_1(\rho) \tilde{c}_2(\rho)) \leq \frac{1}{k+1}$, contradicting our assumption.  

 \begin{figure*}
   \labellist 
  \small\hair 2pt
   \pinlabel $x_2^{2n-1}$ at 140 190 
   \pinlabel $x_1^{2n}$ at  200 190 
    \pinlabel $\tilde{\rho}(b_i)(z^{1})$ at 220 250
    \pinlabel $z^{2n+1}$ at 310 250
     \pinlabel $x_1^{2n+1}$ at  320 190 
   \pinlabel $x_2^{2n}$ at 380 190
  \pinlabel $x_2^{2n+1}$ at 430 190 
    
   \endlabellist
  \centerline{
    \mbox{\includegraphics[width=4in]{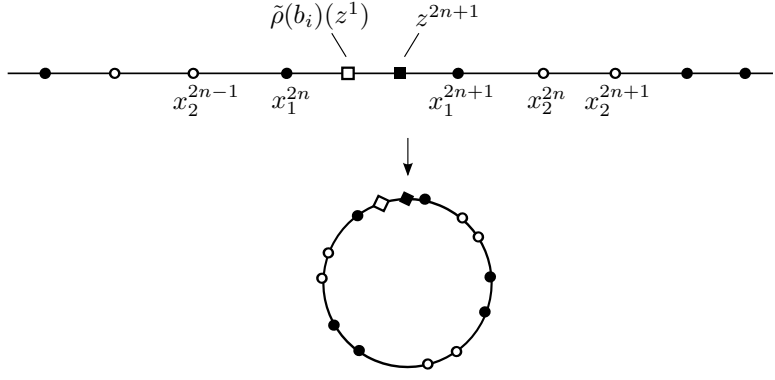}}}
 \caption{\small If $\tilde{\rho}(b_i)(z^{1}) \leq z^{2n +1}$, then there is an extra alternating pair on $S^1$}
  \label{zj figure}
  \end{figure*}

The right side of \eqref{zclaim2} is proved in exactly the same way; we omit the details.    

\end{proof}

To finish the proof of Proposition \ref{int Ni prop}, assume for contradiction that $\rott(\tilde{c}_1(\rho)\tilde{c}_2(\rho)) \geq 1/k$.  Then Claim \ref{Z claim} implies that 
\begin{align*} 
& Z^{2j + 2n +1} < \tilde{\rho}(b_i)(Z^{2j + 1}) <  Z^{2j + 2n + 2} \text{ and } \\
& Z^{2j + 2n -1} < \tilde{\rho}(b_i)(Z^{2j}) <  Z^{2j + 2n}.
\end{align*}

Let $z^m \in Z^m$ be a \emph{stable fixed point} in the sense of Lemma \ref{stable fix points lemma}.  We may also choose $z^m$ such that $z^{m+2k} = z^m +1$.  Since equations \eqref{zclaim1} and \eqref{zclaim2} hold, we have in particular
\begin{align*} 
& z^{2j + 2n +1} < \tilde{\rho}(b_i)(z^{2j + 1}) <  z^{2j + 2n + 2} \text{ and } \\
& z^{2j + 2n -1} < \tilde{\rho}(b_i)(z^{2j}) <  z^{2j + 2n}.
\end{align*} 
Lemma \ref{stable fix points lemma} now implies that, by taking $\rho'$ sufficiently close to $\rho$, we can find fixed points $z^{m}(\rho') \in \fix(\tilde{c}_1(\rho'))$ as close as we like to $z^m$.  In particular, there is an open neighborhood $U_i$ in $\Hom(\Gamma_g, \Homeo_+(S^1))$, such that for all $\rho' \in U_i$, there are points $z^{m}(\rho') \in \fix(\tilde{c}_1(\rho'))$ satisfying
\begin{align} 
& z^{2j + 2n +1}(\rho') < \tilde{\rho'}(b_i)(z^{2j + 1}(\rho')) <  z^{2j + 2n + 2}(\rho') \text{ and } \label{final z eq1} \\
& z^{2j + 2n -1}(\rho') < \tilde{\rho'}(b_i)(z^{2j}(\rho') <  z^{2j + 2n}(\rho') .\label{final z eq2}
\end{align}
(To get the existence of an open neighborhood $U_i$ we are using the fact that $z^{m+2n} = z^m +1$ and that all $\tilde{c}_1(\rho')$ commute with integer translations, so all equations above really only specify a finite set of inequalities).  

Let $X_1(\rho') \subset \fix(c_1(\rho'))$ be the projection of $\bigcup \limits_{m\in \Z} z^m(\rho')$ to $S^1 = \R/\Z$.  
The fact that $x_2^{2j-1} < Z^{2j} < \tilde{\rho}(b_i)^{-1}(x_1^{2j+2n}$ and $\tilde{\rho}(b_i)^{-1}(x_1^{2j+2n+1}) < Z^{2j+1} <  x_2^{2j}$ implies that $X_1(\rho')$ is of cardinality $2k$.  Together with equations \eqref{final z eq1} and \eqref{final z eq2} above, this also implies that properties ii) and iii) of a \emph{good fixed set} hold for the set 
$X_1(\rho')$.  Thus $X_1(\rho')$ is a good fixed set for $c_1(\rho')$, which is what we wanted to show.  

This concludes the proof of Proposition \ref{int Ni prop}, showing that $\rho$ is an interior point of $N_i$. 

\end{proof}


\subsection{Rotation numbers of products of commutators} \label{comm product subsec}

We continue to work with our maximal $\PSLk$ representation $\rho_0$ from the previous section.  It will be convenient to introduce a little more notation.  

\begin{notation}
Given any $\rho \in \Hom(\Gamma_g, \Homeo_+(S^1)$, let $c_i(\rho)$ denote the lifted commutator $[\tilde{\rho}(a_i), \tilde{\rho}(b_i)] \in \Homeo_\Z(\R)$, and let $R(\rho)$ denote the product $\rott( c_1(\rho) c_2(\rho) ...c_{g-1}(\rho))$. 
\end{notation}

Similar to what we did in Section \ref{technical subsec}, our strategy here is to define a \emph{good} representation to be one that ``looks like" a maximal $\PSLk$ representation (at least on the level of the combinatorial data of some periodic points for certain elements), determine the possible values of $R(\rho)$ for ``good" representations $\rho$, and then find interior points of the set of good representations.  

\begin{definition} 
Say that a representation $\rho \in \Hom(\Gamma_g, \Homeo_+(S^1))$ is \emph{good} if the following two conditions hold. 
\begin{enumerate}[i)]
\item $\rott(\tilde{c}_i(\rho)) = 1/k$ for all $i \in \{1, 2, ... g-1\}$, and
\item For each $i$, there is a periodic orbit $X_i(\rho)$ of $c_i(\rho)$ such that the lifts $\tilde{X}_i(\rho)$ can be ordered 
$$\ldots  \, x_1^j(\rho) < x_2^j(\rho) < \ldots  < x_n^j(\rho) < x_1^{j+1}(\rho)   \ldots  $$
\end{enumerate}
\end{definition}

Let $N_0$ be the closure of the set of good representations in $\Hom(\Gamma_g, \Homeo_+(S^1))$.  Note that $\rho_0$ is a good representation, and that for any good representation $\rho$ Example 4.3 implies that 
\begin{equation} \label{eq N0}
\rott(\tilde{c}_1(\rho)\tilde{c}_2(\rho)...\tilde{c}_{g-1}(\rho)) = \rott([\tilde{\rho}(a_1), \tilde{\rho}(b_1)]...[\tilde{\rho}(a_{g-1}), \tilde{\rho}(b_{g-1})]) \leq \frac{2g-3}{2}
\end{equation} 
Moreover, \eqref{eq N0} holds not only for good representations but for all $\rho \in N_0$.  

Let $N$ be the closure of $\bigcap \limits_{i=0}^g N_i$. 

\begin{proposition}[Interior points of $N$]  \label{int N prop}
Let $\rho \in N$.  If $\rott(\tilde{c}_1(\rho)\tilde{c}_2(\rho)...\tilde{c}_{g-1}(\rho)) = \frac{2g-3}{2}$ 
holds, then $\rho$ is an interior point of $S$.  
\end{proposition} 

\begin{proof}
Suppose that $\rho \in N$ satisfies $\rott(\tilde{c}_1(\rho)\tilde{c}_2(\rho)...\tilde{c}_{g-1}(\rho)) = \frac{2g-3}{2}$.  
Let $\rho_n$ be a sequence of representations in $\bigcap \limits_{i=0}^g N_i$ approaching $\rho$.  Then $\rho_n \in N_0$, so we may choose periodic orbits $X_i(\rho_n)$ for $c_i(\rho_n)$ as in the definition of ``good".  After passing to a subsequence, we may assume that each $X_i(\rho_n)$ converges to a $k$-tuple of points $X_i(\rho)$ with lifts ordered 
$$\ldots  \, x_1^j(\rho) \leq x_2^j(\rho) \leq \ldots  \leq x_n^j(\rho) \leq x_1^{j+1}(\rho)   \ldots  $$
If equality holds in any of these inequalities, then Example 4.4 implies that $\rott(\tilde{c}_1(\rho)\tilde{c}_2(\rho)...\tilde{c}_{g-1}(\rho)) < \frac{2g-3}{2}$, contradicting our assumption, so all points $x_i^j(\rho)$ must be distinct.  

Since $\rott(\tilde{c}_i(\rho)) = 1/k$, by Proposition \ref{int Ni prop} there is an open neighborhood $U_i$ of $\rho$ in $\Hom(\Gamma_g, \Homeo_+(S^1))$ contained in $N_i$.  The intersection $U$ of the sets $U_i$ is an open neighborhood of $\rho$ contained in $\bigcap \limits_{i=1}^g S_i$.  We claim that there is an open neighborhood of $\rho$ contained in $N_0$ as well.  

To see this, suppose for contradiction that we can find a sequence of representations $\eta_n$ in $\Hom(\Gamma_g, \Homeo_+(S^1)) \setminus S_0$ approaching $\rho$.  Without loss of generality, we may assume all $\eta_n$ lie in $U$.  Then each $c_i(\eta_n)$ has a periodic orbit $X_i(\eta_n)$ and after passing to a subsequence we may assume that each $X_i(\eta_n)$ converges to a $k$-tuple $Y_i$, which will be a periodic orbit for $c_i(\rho)$.  Now Corollary 4.8 implies that the periodic orbits $Y_i$ have the same combinatorial structure as the $X_i(\rho)$ in the definition of ``good", and hence so do the sets $X_i(\eta_n)$ for $n$ sufficiently large.  Thus, for large $n$, the representation $\eta_n$ is good, contradicting our assumption.  

\end{proof}

\subsection{Rotation rigidity for generators} \label{rot rig subsec}

We can now show that the rotation number of a single element of our standard generating set is constant on the connected component of $\rho_0$.  

\begin{proposition}[\textbf{Rotation rigidity for a single generator}] \label{gen rig prop}
Let $N$ be the set defined in Section \ref{comm product subsec}, and let $N'$ be the connected component of $N$ containing $\rho_0$.    Then $N'$ is a connected component of $\Hom(\Gamma_g, \Homeo_+(S^1))$ and $\rot(\rho(a_g)) = 0$ for all $\rho \in N'$. 
\end{proposition}

\begin{proof}  
By definition, $N$ is an intersection of closed sets, so $N'$ is closed.  To show that $N'$ is also open, consider any $\rho \in N'$.  Since $\rho$ lies in the same component of $\Hom(\Gamma_g, \Homeo_+(S^1))$ as $\rho_0$, we have $\euler(\rho) = \euler(\rho_0) = \frac{2g=2}{k}$.  It follows from Lemma 3.1 and the definition of Euler number that
$$ \rott \left( [\rho(a_1), \rho(b_1)] \ldots [\rho(a_{g-1}), \rho(b_{g-1})]\right) + \rott[\rho(a_g), \rho(b_g)] = \frac{2g-2}{k}$$
From this and the inequalities 
$R_i(\rho) \leq 1/k$ for $\rho \in N_i$, and $R(\rho) \leq \frac{2g-3}{k}$ for $\rho \in N_0$, 
we can conclude that $\rott[\rho(a_g), \rho(b_g)] = 1/k$ and $\rott \left( [\rho(a_1), \rho(b_1)] \ldots [\rho(a_{g-1}), \rho(b_{g-1})]\right) = \frac{2g-3}{k}$.   Thus, by Propositions \ref{int Ni prop} and \ref{int N prop}, there is an open neighborhood of $\rho$ contained in $N$ and hence also in $N'$.   This proves that $N'$ is a connected component of $\Hom(\Gamma_g, \Homeo_+(S^1))$.   Moreover, we have also just shown that $\rott[\rho(a_g), \rho(b_g)] = 1/k$ holds for all $\rho \in N'$.  
Assume now for contradiction that $\rot(\rho(a_g)) \neq \rot(\rho_0(a_g))$ for some $\rho \in N'$.  By continuity of $\rot$, there exists a representation $\rho' \in N'$ such that $\rot(\rho'(a_1))$ is either irrational or is rational of the form $p/q$ with $q>k$.  In either case, Lemma 2.4 implies that $\rott[\rho'(a_g), \rho'(b_g)] < 1/k$, a contradiction.   

\end{proof}

\subsection{Reduction to the case $\rot(\rho_0(a_i)) = 0$: a Euclidean algorithm trick} \label{Euc alg subsec}

Assuming now that $\rot(\rho_0(a_i)) = \frac{m_i}{k}$ for some $m \neq 0$, we show how to define a continuous map $\Phi: \Hom(\Gamma_g, \Homeo_+(S^1)) \to \Hom(\Gamma_g, \Homeo_+(S^1))$ such that $\Phi(\rho_0(a_i))$ is a maximal $\PSLk$ representation with $\rot(\Phi(\rho_0(a_i)))= 0$. Moreover, $\Phi$ will have additional properties that will let us translate the results of Sections \ref{technical subsec} and \ref{rot rig subsec} to this new case.  The key to reducing the rotation number of $\rho_0(a_i)$ is an \emph{Euclidean algorithm} for rotation numbers in commutators, introduced in Proposition \ref{Euclidean alg prop}

We start with a characterization of pairs $(f,g) \in \PSLk \times \PSLk$ such that $f = \rho(a_i)$ and $g= \rho(b_i)$ for a pair $(a_i, b_i)$ in a standard generating set and a maximal $\PSLk$ representation $\rho$.  

\begin{definition}[crossed pair]
Call a pair $(f, g) \in \PSLk \times \PSLk$  \emph{crossed} if the projections $\bar{f}$ and $\bar{g}$ of $f$ and $g$ to $\PSL(2,\R)$ are hyperbolic elements with intersecting (i.e. crossed) axes, with intersection number $-1$.  
\end{definition} 

\begin{proposition}
Let $(f,g)$ be a crossed pair.  Then there exists a maximal $\PSLk$ representation $\rho_0: \Gamma_g \to \PSLk$ such that $f = \rho_0(a_i)$ and $g= \rho_0(b_i)$ for a pair $(a_i, b_i)$ of elements in a standard generating set. 
\end{proposition}
 
\begin{proof} Let $(f,g)$ be a crossed pair.  Since maximal $\PSLk$ representations are lifts of injective Fuchsian representations to a $k$-fold cover, it suffices to exhibit an injective, Fuchsian representation $\nu: \Gamma_g \to \PSL(2,\R)$ such that $\nu(a_1) = \bar{f}$ and $\nu(b_1) = \bar{g}$.   It follows from standard theory of hyperbolic structures on surfaces that such a representation exists whenever $\bar{f}$ and $\bar{g}$ are hyperbolic elements with crossed axes.  In particular, one can deduce this from the usual description of Fenchel-Nielsen coordinates on Teichmuller space.  The assumption that the axes have intersection number $-1$ ensures that the Euler number of the representation $\nu$ will have Euler number $2g-2$ rather than $-2g+2$.   

\end{proof}

\begin{remark} 
For a less sophisticated approach, we remark that our proof Proposition of \ref{int Ni prop} in the case $\rot(\rho_0(a_i)) = 0$ did not use the fact that $\rho_0(a_i)$ and $\rho_0(a_i)$ were the images of a standard pair of generators under a maximal $\PSLk$ representation -- we only used the combinatorial structure (i.e. the ordering) of the points of $\fix(c_1(\rho_0))$ and  $\fix(c_2(\rho_0))$. In particular, we needed that $\fix(c_1(\rho_0)) = \fix(\rho_0(a_i))$ was a \emph{good fixed set}.  Now it is easy to show that for any crossed pair $(f, g)$ with $\rot(f) = 0$, the sets $\fix(f)$ and $g \fix(f)$ have the same combinatorial structure as $\fix(c_1(\rho_0))$ and  $\fix(c_2(\rho_0))$, and in particular $X_1 = \fix(f)$ satisfies the properties of a good fixed set for $f$, with $g$ playing the role of $\rho(b_i)$ in the definition of ``good fixed set". 
\end{remark}

We now show how to use one crossed pair to build another.  
\begin{lemma} \label{crossed axes lemma}
Let $(f, g)$ be a crossed pair.  Then $(fg, g)$ and $(f, gf)$ are crossed pairs also. 
\end{lemma}

\begin{proof} 
Let $f$ and $g$ be crossed, with projections $\bar{f}$ and $\bar{g} \in \PSL(2,\R)$.  
The dynamics of the action of $\bar{f}$ and $\bar{g}$ on $S^1$ is as in Figure \ref{fg fig} below.  Let $I \subset S^1$ be the closed interval bounded by the attracting fixed points of $\bar{f}$ and $\bar{g}$, and $J$ the closed interval bounded by the repelling fixed points, as in the figure.
Then $\bar{f}\bar{g}(I) \subset I$ and $(\bar{f}\bar{g})^{-1}(J) \subset J$.  It follows that $\bar{f}\bar{g}$ has an attracting fixed point in $I$ and a repelling fixed point in $J$, so its axis crosses the axis of $\bar{g}$.  The same argument shows that the axis of $\bar{g}\bar{f}$ crosses the axis of $\bar{g}$.  

\end{proof}

 \begin{figure*}
   \labellist 
  \small\hair 2pt
   \pinlabel $f$ at 45 50 
   \pinlabel $g$ at 110 100
   \pinlabel $I$ at 165 80
   \pinlabel $J$ at 60 -5
   \endlabellist
  \centerline{
    \mbox{\includegraphics[width=1.5in]{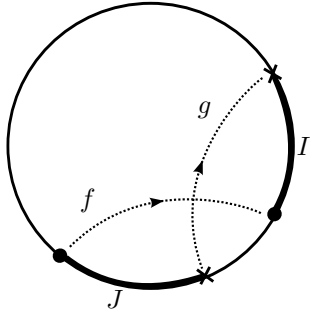}}}
 \caption{\small A pair of crossed homeomorphisms}
  \label{fg fig}
  \end{figure*}

Our final lemma is an elementary result on rotation numbers of positive words.  Recall that we use the terminology \emph{word in a and b} to mean a word $w(a,b)$ in the letters $a$ and $b$, (and not $a^{-1}$ or $b^{-1}$), in other words, $w(a,b)$ is an element of the positive semigroup generated by $a$ and $b$.    

\begin{lemma} \label{rot 0 lemma}
Let $(f, g)$ be a crossed pair, with $\rot(f) = \rot(g) = 0$.  Then any word $w = w(f,g)$ has $\rot(w) = 0$.  
\end{lemma}

\begin{proof}
Let $(f,g)$ be a crossed pair with projections $\bar{f}$ and $\bar{g} \in \PSL(2,\R)$.  Let $I \subset S^1$ be an interval as in Figure \ref{fg fig}, so $\bar{f}(I) \subset I$ and $\bar{g}(I) \subset I$.  Lift $I$ to a single connected interval $I'$ in the $k$-fold cover of $S^1$.  Since $\rot(f) = \rot(g) = 0$, it follows that $f(I') \subset I'$ and $g(I') \subset I'$,  Thus, $w(f,g)(I') \subset I'$ and so $w$ has a fixed point in $I'$.  

\end{proof}

Now we can prove the main result of this section.  

\begin{proposition}[\textbf{Euclidean algorithm for rotation numbers}] \label{Euclidean alg prop}
Let $(a,b) \in \PSLk \times \PSLk$ be a crossed pair.  There exist words $u = u(a,b)$ and $v = v(a,b)$ such that 
\begin{enumerate}[i)]
\item $[u, v] = [a,b]$ 
\item $\rot(u) = 0$
\item $u$ and $v$ are a crossed pair.  
\end{enumerate}
\end{proposition}

\begin{proof}
We will make repeated use of the following elementary algebraic computation, which applies to any commutator.  
\begin{equation}  \label{comm eq}
[f,g] = [f, gf^n] = [fg^n, g] \text{ for any } n \in \Z
\end{equation}

Now let $m = k \rot(a)$ and $n = k \rot(b)$.   Let $r: S^1 \to S^1$ denote the order $k$ rigid rotation with rotation number $\frac{1}{k}$, which commutes with $a$ and $b$. Let $a' = a r^{k-m}$ and $b' = b r^{k-n}$. Note that  $\rot(a') = \rot(b') = 0$, and that $(a', b')$ is a crossed pair in $\PSLk$.   

As a warm up for the rest of the proof, we do a simple computation.  Let $\alpha_1, \beta_1, ... \alpha_s, \beta_s$ be integers.  Then
$$a^{\alpha_1} b^{\beta_1}...a^{\alpha_s} b^{\beta_s} = (a')^{\alpha_1} (b')^{\beta_1}... (a')^{\alpha_s} (b')^{\beta_s} r^{(m(\alpha_1 +...+\alpha_s) + n( \beta_1 + ...+\beta_s))}$$  
By Lemma \ref{rot 0 lemma}, $\rot((a')^{\alpha_1} (b')^{\beta_1}... (a')^{\alpha_s} (b')^{\beta_s}) = 0$, so 
\begin{equation} \label{warm up eq}
\rot(a^{\alpha_1} b^{\beta_1}...a^{\alpha_s} b^{\beta_s} ) = \frac{m(\alpha_1 +...+\alpha_s) + n( \beta_1 + ...+\beta_s)}{k}.
\end{equation}

This kind of computation lets us use the Euclidean algorithm to reduce the rotation number of $a$ while preserving the commutator $[a,b]$ as follows.  
First, use the (standard) Euclidean algorithm to produce a sequence of pairs of integers
$$(m, n) = (m_0, n_0), \, (m_0, n_1), \,(m_1, n_1), \,(m_1, n_2), ...$$
 with $n_i = n_{i-1} -q _i m_{i-1}$ and $m_i = m_{i-1} - p_i  n_{i}$, terminating after $d$ steps with a pair either of the form $(m_d,0)$ or $(0, n_d)$.  
Choose $\bar{q}_i >0$ and $\bar{p}_i>0$ such that $\bar{q}_i \equiv -q_i \mod k$, and $\bar{p}_i \equiv -p_i \mod k$, and consider the sequence of pairs of in $\PSLk$
$$(f_0, g_0), \,(f_0, g_1), \,(f_1, g_1), \,(f_1, g_2)... $$
defined recursively by setting $(f_0, g_0) = (a,b)$ and
\begin{align*} & g_i = g_{i-1}f_{i-1}^{\bar{q}_i} \\
& f_i = f_{i-1} g_{i-1}^{\bar{p}_i}
\end{align*}
The sequence terminates at step $d$ with in a pair of words $(u, v) = (u(a,b), v(a,b))$.  
The reader may find it instructive to look at the first few terms of the sequence:
$$(a, b), \,(a, b a^{\bar{q}_1}), \,(a (ba^{\bar{q}_1})^{\bar{p}_1}, ba^{\bar{q}_1}),... $$

Our recursive definition of $f_i$ and $g_i$ together with Equation \eqref{comm eq} implies that for each $i$ we have
$$[f_i, g_i] = [f_i, g_{i+1}] = [f_{i+1}, g_{i+1}].$$ 
It follows that $[u,v] = [a,b]$.  
Moreover, Lemma \ref{crossed axes lemma} implies (inductively) that each pair in the sequence is a \emph{crossed pair}.   
Finally, the calculation in our warm-up (equation \eqref{warm up eq}) shows that $\rot(f_i) = m_i/k$  and $\rot(g_i) = n_i/k$ (recall that rotation numbers take values in $\R / \Z$). Thus, the final pair $(u,v)$ satisfies either $\rot(u) = 0$ or $\rot(v) = 0$.   If $\rot(u) = 0$, we are done.  If instead $\rot(v) = 0$, replace $(u, v)$ with the pair $(uvu^{k-1}, vu^k)$.  Lemma \ref{crossed axes lemma} implies first that $(u, vu^k)$ is a crossed pair, and then so is $(uvu^{k-1}, vu^k)$.  The commutators satisfy $[uvu^{k-1}, vu^k] = [u,v] = [a,b]$  and now $\rot(uvu^{k-1}) = \rot(v) = 0$, as desired.

\end{proof}

We can now carry out the reduction to the case $\rot(\rho_0(a_i)) = 0$, justifying our assumption in Sections \ref{technical subsec} through \ref{rot rig subsec}. 
 
Let $\rho_0$ be any maximal $\PSLk$ representation.  
Since each pair $(\rho_0(a_i), \rho_0(b_i))$ is a crossed pair, we may use Proposition \ref{Euclidean alg prop} to produce words $u_i(a,b)$ and $v_i(a,b)$ in the letters $a$ and $b$ such that -- letting $u_i(\rho)$ denote $u_i(\rho(a_i), \rho(b_i))$ and $v_i(\rho)$ denote $v_i(\rho(a_i), \rho(b_i))$ -- we have 
\begin{enumerate}[i)]
\item  $[u_i(\rho) , v_i(\rho)] = [\rho(a_i), \rho(b_i)]$ for any $\rho$, and
\item $\rot(u_i(\rho_0)) = 0$.
\end{enumerate}

Consider the function $\Phi: \Hom(\Gamma_g, \Homeo_+(S^1)) \to \Hom(\Gamma_g, \Homeo_+(S^1))$ defined as follows.  It is enough to specify the image of standard generators, and we set $\Phi(\rho)(a_i) = u_i(\rho)$ and $\Phi(\rho)(b_i) = v_i(\rho)$.  This function is continuous, and the property i) above implies that it is well defined.   Let $\eta_0 = \Phi(\rho_0)$.  Then the image of $\eta_0$ lies in $\PSL^k$, and property i) implies that $\euler(\eta_0) = \euler(\rho_0) = \frac{2g-2}{k}$.   Since $\rot(\eta_0(a_i)) = \rot(u_i(\rho_0)) = 0$ for all $i$, the results of Sections \ref{technical subsec} through \ref{rot rig subsec}  above apply here to show that $\rott[\eta(a_g), \eta(b_g)] = 1/k$ for all $\eta$ in the same connected component as $\eta_0$ of $\Hom(\Gamma_g, \Homeo_+(S^1))$.   

Now suppose that $\rho$ is a representation in the same component as $\rho_0$.  Then $\Phi(\rho)$ lies in the same component as $\eta_0$, so  
$$\rott[\rho(a_g), \rho(b_g)] = \rott[\Phi(\rho)(a_g), \Phi(\rho)(b_g)] = \rott[\eta(a_g), \eta(b_g)] = 1/k.$$
Thus, $\rott[\rho(a_g), \rho(b_g)] = 1/k$ for all $\rho$ in the same connected component as $\rho_0$.  The argument 
from the end of Proposition \ref{gen rig prop}
now applies to show that $\rot(\rho(a_g)) = \rot(\rho_0(a_g))$ for all $\rho$ in the connected component of $\rho_0$ in $\Hom(\Gamma_g, \Homeo_+(S^1))$.

In summary, we have just shown the following. 

\begin{proposition}[\textbf{Rotation rigidity for generators, general version}] \label{gen rig prop II}
Let $\rho_0$ be a maximal $\PSLk$ representation, and $\{a_1, b_1, ... a_g, b_g\}$ a standard generating set for $\Gamma_g$.  Let $X \subset \Hom(\Gamma_g, \Homeo_+(S^1))$ be a connected component containing a maximal $\PSLk$ representation $\rho_0$. 
Then $\rot(\rho(a_g)) = \rot(\rho_0(a_g))$ for all $\rho \in X$.
\end{proposition}

\subsection{Finishing the proof of Theorem \ref{rot rig thm}}

Theorem \ref{rot rig thm} will now follow from Proposition \ref{gen rig prop II} using a covering trick due to Matsumoto in \cite{Matsumoto}.  Let $X$ be a component of $\Hom(\Gamma_g, \Homeo_+(S^1))$ containing a maximal $\PSLk$ representation $\rho_0$, and let $\rho \in X$.   
Let $\gamma \in \Gamma_g$, and let $\alpha$ be a curve on $\Sigma_g$ representing $\gamma$.   If $\alpha$ is a nonseparating simple closed curve, then we may include $\gamma$ into a standard generating set for $\Gamma_g$, in which case Proposition \ref{gen rig prop II} implies that $\rot(\rho(\gamma)) = \rot(\rho_0(\gamma))$.  

If $\alpha$ is not a simple closed curve, we may take a finite cover $\Sigma_{g'}$ of $\Sigma_g$ such that $\alpha$ lifts to a nonseparating simple closed curve $\beta$ (such a cover always exists -- Scott's theorem in \cite{Scott} implies that $\alpha$ can be lifted to a simple closed curve in some finite cover, and taking a further cover we can ensure that the lift is nonseparating).  Let $a_{g'} \in \pi_1(\Sigma_{g'})$ represent this lift, and include $a_{g'}$ in a standard generating set $\{a_1, b_1 ...  a_{g'}, b_{g'}\}$ for $\Sigma'$.  
If the degree of the cover $\Sigma_{g'} \overset{\pi}\rightarrow \Sigma_g$ is $m$, then $g' = m(g-1)+1$.  

Consider now the (continuous) function $\pi^*: \Hom(\Gamma_g, \Homeo_+(S^1)) \to \Hom(\Gamma_{g'}, \Homeo_+(S^1))$ defined by $\rho \mapsto \rho \circ \pi_*$.   The Euler number is multiplicative with respect to covers, so we have  
$$\euler(\pi^*(\rho)) = m\euler(\rho) = \frac{m(2g-2)}{k} = \frac{2g' -2}{k}$$  
In particular, $\pi^*(\rho_0)$ is also a maximal $\PSLk$ representation.  Since $\pi^*(\rho)$ lies in the same component of $\Hom(\Gamma_{g'}, \Homeo_+(S^1))$ as $\pi^*(\rho_0)$, Proposition \ref{gen rig prop II} implies that  
$$\rot(\pi^*(\rho_0(a_{g'}))) = \rot(\pi^*(\rho(a_{g'}))).$$  Since $\pi_*(a_{g'}) = \alpha$, we conclude that $\rot(\rho_0(\alpha)) = \rot(\rho(\alpha))$ as desired.  
This concludes the proof of Theorem \ref{rot rig thm}.  

\qed

\section{Proof of Theorem \ref{main thm}} \label{main pf sec}

Theorem \ref{main thm} will now follow from Theorem \ref{rot rig thm}.  We recall the statement here.  
\medskip 

\noindent \textbf{Theorem \ref{main thm}.} For each nontrivial divisor $k$ of $2g-2$, there are at least $k^{2g} +1$ components of $\Hom(\Gamma_g, \Homeo_+(S^1))$ consisting of representations with Euler number $\frac{2g-2}{k}$.   

\noindent In particular, two representations into $\PSLk$ that lie in different components of $\Hom(\Gamma_g, \PSLk)$ necessarily lie in different components of $\Hom(\Gamma_g, \Homeo_+(S^1))$.
\smallskip

\begin{proof}
Let $k$ be a nontrivial divisor of $2g-2$.  As explained in the proof of Proposition \ref{maximal k reps prop}, representations $\rho: \Gamma_g \to \PSLk$ with Euler number $\euler(\rho) = \frac{2g-2}{k}$ are precisely the lifts of faithful Fuchsian representations $\nu: \Gamma_g \to \PSL(2, \R)$.   Each representation $\nu$ has $k^{2g}$ lifts; these can be distinguished by reading the rotation numbers of each of a standard set of generators, which may take any value in $\{0, \frac{1}{k}, ... , \frac{k-1}{k}\}$.  

Goldman's theorem (Theorem \ref{goldman thm}) states that the rotation numbers of the generators are a complete invariant of connected components of $\Hom(\Gamma_g, \PSLk)$ -- there are $k^{2g}$ components of $\Hom(\Gamma_g, \PSLk)$, distinguished by the rotation numbers of a standard set of generators.   Thus, if $\rho_1$ and $\rho_2$ are representations that lie in different components of $\Hom(\Gamma_g, \PSLk)$, then the rotation number of some generator differs under $\rho_1$ and $\rho_2$, and so by Theorem  \ref{rot rig thm}, $\rho_1$ and $\rho_2$ must lie in different components of $\Hom(\Gamma_g, \Homeo_+(S^1))$.

Finally, to conclude that $\Hom(\Gamma_g, \Homeo_+(S^1))$ has at least $k^{2g} +1$ connected components consisting of representations with Euler number $\frac{2g-2}{k}$, it suffices to exhibit a representation $\rho: \Gamma_g \to \Homeo_+(S^1)$ with $\euler(\rho) = \frac{2g-2}{k}$ and $\rho(a_g) \neq \frac{m}{k}$ for any integer $m$.  Such a representation can be constructed as follows.  Take a representation $\mu: \Gamma_{g-1} \to \Homeo_+(S^1)$ with $\euler(\mu) = \frac{2g-2}{k}$ and extend this to a representation $\rho: \Gamma_g \to \Homeo_+(S^1)$ as follows:  If  $\{a_1, b_1, ... a_{g-1}, b_{g-1} \}$ is a standard generating set for $\Gamma_{g-1}$, then we can take $\{a_1, b_1, ... a_{g}, b_{g} \}$ a standard generating set for $\Gamma_g$, and define $\rho$ on generators by 
$$ \begin{array}{lll}
 \rho(a_i) = \nu(a_i), & \rho(b_i) = \nu(b_i), \,\, &  \text{ for } i = 1, 2, ...,g-1 \\ 
 \rho(a_g) = r_\alpha,  & \rho(b_g) = r_\beta  &
\end{array}
$$
where $r_\alpha$ and $r_\beta$ are rigid rotations by arbitrary angles $\alpha$ and $\beta$, with $2 \pi k \beta \notin \Z$.  Since $[r_\alpha, r_\beta] = \id$, we have 
$$[\tilde{\rho}(a_1), \tilde{\rho}(b_1)]...[\tilde{\rho}(a_{g-1}), \tilde{\rho}(b_{g-1})] [\tilde{r}_\alpha, \tilde{r}_\beta] = [\tilde{\rho}(a_1), \tilde{\rho}(b_1)]...[\tilde{\rho}(a_{g-1}), \tilde{\rho}(b_{g-1})],$$
a translation by $\frac{2g-2}{k}$.  Thus, $\rho$ is a representation and has Euler number $\euler(\rho) = \frac{2g-2}{k}$. 

\end{proof}

\section{Semi-conjugate representations: Theorem \ref{semiconj thm}} \label{semiconj sec}

In this section we recall the notion of \emph{semi-conjugacy} and prove Theorem \ref{semiconj thm}.  The reader should note that use of the term ``semi-conjugacy" for representations to $\Homeo_+(S^1)$ in the existing literature is inconsistent.  We will use the following definition, which appears in \cite{Ghys groups acting}.  

\begin{definition}[Degree 1 monotone map]
A map $h: S^1 \to S^1$ is called a \emph{degree 1 monotone map} if it is continuous and admits a lift $\tilde{h}: \R \to \R$, equivariant with respect to integer translations, and nondecreasing on $\R$.  
\end{definition}

\begin{definition}[Semi-conjugacy of representations]
Let $\Gamma$ be any group.  Two representations $\rho_1$ and $\rho_2$ in $\Hom(\Gamma, \Homeo_+(S^1))$ are \emph{semi-conjugate} if there is a degree one monotone map $h: S^1 \to S^1$ such that $\rho_1(\gamma) \circ h = h \circ \rho_2(\gamma)$ for all $\gamma \in \Gamma$.    
\end{definition}

Semi-conjugacy is not an equivalence relation because it is not symmetric.  However, there is a relatively simple description of representations that lie in the same class under the equivalence relation generated by semi-conjugacy.  This description is due to Calegari and Dunfield.  

\begin{proposition}[Definition 6.5 and Lemma 6.6 of \cite{CD}] \label{CD prop}
Two representations $\rho_1$ and $\rho_2$ in $\Hom(\Gamma, \Homeo_+(S^1))$ lie in the same semi-conjugacy class if and only if there is a third representation $\rho: \Gamma \to \Homeo_+(S^1)$ and degree one monotone maps $h_1$ and $h_2$ of $S^1$ such that $\rho_i \circ h_i = h_i \circ \rho$.  
 \end{proposition}

In \cite{Ghys semiconj}, Ghys showed that two representations in $\Hom(\Gamma, \Homeo_+(S^1))$ lie in the same semi-conjugacy class if and only if they have the same \emph{bounded integer Euler class} in $H^2_b(\Gamma; \Z)$.  
Matsumoto later translated this into a condition defined purely in terms of (lifted) rotation numbers.  Before we state this condition, note that for any two elements $a$, $b \in \Homeo_+(S^1)$ with lifts $\tilde{a}$ and $\tilde{b}$ in $\Homeo_\Z(\R)$, the number 
$$\tau(a,b):= \rott(\tilde{a} \tilde{b}) - \rott(\tilde{a}) -\rott(\tilde{b})$$
 does not depend on the choice of lifts $\tilde{a}$ and $\tilde{b}$.   

\begin{proposition}[Matsumoto, \cite{Matsumoto num}] \label{matsumoto prop}
Let $\Gamma$ be a group with generating set $\{\gamma_i\}$.  Two representations $\rho_1$ and $\rho_2$ in $\Hom(\Gamma, \Homeo_+(S^1))$ lie in the same semi-conjugacy class if and only if the following two conditions hold
\begin{enumerate}[i)]
\item Each generator $\gamma_i$ of $\Gamma$ satisfies $\rot(\rho_1(\gamma)) = \rot(\rho_2(\gamma))$. 
\item Each pair of elements $\gamma$ and $\gamma'$ in $\Gamma$ satisfies $\tau(\rho_1(\gamma), \rho_1(\gamma')) = \tau (\rho_2(\gamma), \rho_2(\gamma'))$.  
\end{enumerate}
\end{proposition}

As a corollary of Matsumoto's condition we have the following.  Recall that for $\gamma \in \Gamma$, we let $\rot_\gamma: \Hom(\Gamma, \Homeo_+(S^1)) \to \R/\Z$ be defined by $\rot_\gamma(\rho) = \rot(\rho(\gamma))$.  

\begin{corollary} \label{matsumoto cor}
Let $\Gamma$ be any group and $U \subset \Hom(\Gamma, \Homeo_+(S^1))$ a connected set.   If for all $\gamma \in \Gamma$ the function $\rot_\gamma$ is constant on $U$, then $U$ consists of a singe semi-conjugacy class.  
\end{corollary} 

\begin{proof}
By assumption, condition i) of Proposition \ref{matsumoto prop} is automatically satisfied for any two representations in $U$.  
To see that ii) is satisfied, fix any $\gamma$ and $\gamma'$ in $\Gamma$, and define a function $F_{\gamma, \gamma'}: U \times U \to \R$ by 
$$F_{\gamma, \gamma'}(\rho, \nu) = \tau(\rho(\gamma) \rho(\gamma')) - \tau(\nu(\gamma), \nu(\gamma')).$$
Then $F_{\gamma, \gamma'}$ is clearly continuous.  We will show it is integer valued, hence constant.  Since $F_{\gamma, \gamma'}(\rho, \rho) = 0$, it will follow that condition ii) is satisfied.  

That $F_{\gamma, \gamma'}$ is constant comes from the fact that $\rot_\gamma$ and $\rot_{\gamma'}$ are constant on $U$.  By definition, 
$$F_{\gamma, \gamma'}(\rho, \nu) = \rott(\tilde{\rho}(\gamma) \tilde{\rho}(\gamma')) - \rott(\tilde{\nu}(\gamma) \tilde{\nu}(\gamma')) - \rott(\tilde{\rho}(\gamma)) + \rott(\tilde{\rho}(\nu)) -\rott(\tilde{\rho}(\gamma')) + \rott(\tilde{\nu}(\gamma'))$$

Taken mod $\Z$, this expression becomes
$$\rot(\rho(\gamma) \rho(\gamma')) - \rot(\nu(\gamma)\nu(\gamma')) - \rot(\rho(\gamma)) + \rot(\rho(\nu)) -\rot(\rho(\gamma')) + \rot(\nu(\gamma')),$$
which is zero, since $\rot_\gamma$  and $\rot_{\gamma'}$ are constant on $U$.  

\end{proof}

 Combining Theorem \ref{rot rig thm} with Corollary \ref{matsumoto cor}, we conclude that if $X$ is a connected component of  $\Hom(\Gamma_g, \Homeo_+(S^1))$ that contains a maximal  $\PSLk$ representation, then all representations in $X$ lie in the same semi-conjugacy class.   Thus, to finish the proof of Theorem \ref{semiconj thm}, it suffices to prove that representations in the same semi-conjugacy class always lie in the same component of $\Hom(\Gamma_g, \Homeo_+(S^1))$.  In fact, they always lie in the same \emph{path}-component.   

\begin{proposition}
If $\rho_1$ and $\rho_2$ lie in the same semi-conjugacy class, then there is a continuous path $\rho_t$ in $\Hom(\Gamma_g, \Homeo_+(S^1))$ from $\rho_0$ to $\rho_1$.  
\end{proposition}

\begin{proof}
Let $\rho: \Gamma_g \to \Homeo_+(S^1)$ and let $h_i: S^1 \to S^1$ be monotone maps as in Proposition \ref{CD prop} with $h_i \circ \rho = \rho_i \circ h_i$.  We will construct a path from $\rho_1$ to $\rho$ (and the same procedure will work to build a path from $\rho_2$ to $\rho$).  

First, for $t \in [1/2, 1)$ let $\hat{h}_t$ be a path of homeomorphisms of $S^1$ such that $\lim \limits_{t \to 1} \hat{h}_t = h$.  Then $\hat{h}_{1/2}$ is isotopic to the identity, so for $t \in [0, 1/2]$ define $\hat{h}_t$ to be a path from $\id$ to $\hat{h}_{1/2}$.  Now, for all $t \in [0,1)$ we may define $\rho_t$ by 

$$\rho_t(\gamma) = \hat{h}_t \rho(\gamma) \hat{h}_t^{-1}$$ 
Note that $\rho_0 = \rho$, and $\rho_t(\gamma) \to \rho_1(\gamma)$ as $t \to 1$ (using the fact that $h$ is surjective).  Thus, $\rho_t$ gives a continuous path from $\rho_1$ to $\rho$, and the same kind of construction will clearly work for $\rho_2$.

\end{proof}
This concludes our proof of Theorem \ref{semiconj thm}.

\section{Rigidity and Flexibility} \label{sharpness sec}

Say that a representation $\rho \in \Hom(\Gamma_g, \Homeo_+(S^1))$ is \emph{rigid} if the connected component of $\Hom(\Gamma_g, \Homeo_+(S^1))$ containing $\rho$ consists of a single semi-conjugacy class.   (Proposition \ref{matsumoto prop} and Corollary \ref{matsumoto cor} imply that this is equivalent to the functions $\rot_\gamma$ being constant on the component containing $\rho$, for each $\gamma \in \Gamma_g$.)  
Theorem \ref{semiconj thm} states that maximal $\PSLk$ reps are rigid.  We conjecture that these are effectively the \emph{only} rigid representations in $\Hom(\Gamma_g, \Homeo_+(S^1))$.

\begin{conjecture} \label{rig conj}
Suppose that $\rho \in \Hom(\Gamma_g, \Homeo_+(S^1))$ is rigid.  Then $\euler(\rho) = \frac{2g-2}{k}$ for some $k$, and $\rho$ lies in the semi-conjugacy class of a representation with image in $\PSLk$.  
\end{conjecture}

One can make an analogous statement for $\Hom(\Gamma_g, \Diff_+(S^1))$

\begin{conjecture} 
Suppose that $\rho \in \Hom(\Gamma_g, \Diff_+(S^1))$ is rigid, meaning that the connected component of $\Hom(\Gamma_g, \Diff_+(S^1))$ containing $\rho$ consists of a single semi-conjugacy class.  Then $\euler(\rho) = \frac{2g-2}{k}$ for some $k$, and $\rho$ lies in the semi-conjugacy class of a representation with image in $\PSLk$.  
\end{conjecture}

Progress on these conjectures, as well as related work on flexibility and rigidity of representations to $\Hom(\Gamma_g, \Diff_+(S^1))$ is the subject of a forthcoming paper.   For now, we note the following theorem of Goldman, which gives a quick answer to Conjecture \ref{rig conj} in the special case where $\rho$ has image in $\PSLk$.  

\begin{proposition}[Goldman, see the proof of Lemma 10.5 in \cite{Goldman}.]  \label{rot flex prop}
Let $\rho: \Gamma_g \to \PSLk$ satisfy $|\euler(\rho)| < \frac{2g-2}{k}$.  Then 
there is a representation $\nu$ in the same path component of $\Hom(\Gamma, \PSLk)$ as $\rho$, and an element $\gamma \in \Gamma_g$ such that $\rot(\rho(\gamma)) \neq \rot(\nu(\gamma))$.  
\end{proposition}




\vspace{1cm}


\vspace{.5in}

Dept. of Mathematics, University of Chicago.  

5734 University Ave. Chicago, IL 60637.  

\textit{E-mail:} mann@math.uchicago.edu

\end{document}